\documentclass[11pt,a4paper]{article}
 \usepackage[latin9]{inputenc}
\usepackage{amsfonts}
\usepackage{amssymb}
\usepackage{amsthm}
\usepackage[centertags]{amsmath}
\usepackage{newlfont}
\usepackage{graphicx}

\newtheorem{theorem}{Theorem}[section]
\newtheorem{corollary}[theorem]{Corollary}
\newtheorem{definition}[theorem]{Definition}

\newtheorem{proposition}[theorem]{Proposition}
\newtheorem{remark}[theorem]{Remark}
\def\r{\mathbb R}

\title{Invariant surfaces in Euclidean space with a log-linear density}
\author{Rafael L\'opez\footnote{Partially
supported by MEC-FEDER
 grant no. MTM2014-52368-P}\\
 Departamento de Geometr\'{\i}a y Topolog\'{\i}a\\ Instituto de Matem\'aticas (IEMath-GR)\\
 Universidad de Granada\\
 18071 Granada, Spain\\
\texttt{rcamino@ugr.es}}

\date{}

\begin{document}

\maketitle
\begin{abstract}
A $\lambda$-translating soliton  with density vector $\vec{v}$   is a surface in Euclidean space   whose mean curvature $H$ satisfies $2H=2\lambda+\langle N,\vec{v}\rangle$, where $N$ is the Gauss map.  We classify all $\lambda$-translating solitons that are invariant by a one-parameter group of translations and a one-parameter group of rotations.
\end{abstract}

\section{Introduction}
Fix a unit vector $\vec{v}$ in Euclidean space $\r^3$ and $\lambda$ a real number. In this paper we study orientable surfaces $\Sigma$ in $\r^3$ whose mean curvature $H$ satisfies
\begin{equation}\label{eq1}
H(p)=\lambda+\frac{\langle N(p),\vec{v}\rangle}{2},\quad p\in\Sigma,
\end{equation}
  where $N$ is the Gauss map of $\Sigma$. The interest of this equation is due to its relation with manifolds with density. Indeed, consider $\r^3$ with a  positive smooth density function $e^\phi$ , $\phi\in C^\infty(\r^3)$, which serves as a weight for  the volume and the surface area. The first variation  of the area $A_\phi$ with density $e^\phi$ under compactly supported variations and with variation vector field $\xi$ is
$$\frac{d}{dt}{\Big|}_{t=0}A_\phi(t)=-2\int_\Sigma  H_\phi\langle N,\xi\rangle  dA_\phi,$$
where $H_\phi=H-\frac12\frac{d\phi}{dN}$. Then it is immediate that  $\Sigma$ is a critical point of $A_\phi$ for a  given weighted volume if and only if $H_\phi$ is a constant function $H_\phi=\lambda$: see \cite{gr,mo}. In this paper we are interested in the log-linear density $e^\phi$ where
$$\phi:\r^3\rightarrow\r,\quad \phi(q)=\langle q,\vec{v}\rangle,$$
and $\vec{v}$ is a unit fixed vector of $\r^3$. Then $H_\phi=H-\langle N,\vec{v}\rangle /2$ and $H_\phi=\lambda$  is exactly (\ref{eq1}).

\begin{definition} A surface $\Sigma$ in $\r^3$ is called a   $\lambda$-translating soliton if Eq. (\ref{eq1}) holds everywhere. The vector $\vec{v}$   is called  the density vector.
\end{definition}

A particular case of (\ref{eq1}) is when $\lambda=0$,
because the equation $2H=\langle N,\vec{v}\rangle$ appears in the singularity theory of the mean curvature flow, indeed, it   is  the  equation of the limit flow by a proper blow-up procedure near type II singular points (\cite{hs,il,wa}). In the literature, a solution of $H_\phi=0$ is  called  a \emph{translating soliton} of the mean curvature flow. Equation \eqref{eq1} can viewed as a type of prescribed mean curvature equation, in fact,  and in a nonparametric form, the equation $H_\phi=0$   appeared in the classical article of Serrin  \cite[p. 477--478]{se} and it was studied  in the context of the maximum principle.

It is immediate that if we reverse the orientation on a $\lambda$-translating soliton, then we obtain a $-\lambda$-translating soliton. It is also clear that every rigid motion of $\r^3$ that leaves invariant the term $\langle N(p),\vec{v}\rangle$ in (\ref{eq1}) is a transformation that preserves the value of $H_\phi$. This occurs when we consider a translation, a rotation about a straight line parallel to $\vec{v}$ or a reflection about a plane parallel to $\vec{v}$.

Some examples of $\lambda$-translating solitons are:
\begin{enumerate}
\item The case $\lambda=0$ has been widely studied in the literature. Some explicit examples of translating solitons are: a  plane parallel to $\vec{v}$, the grim reaper (a surface of translation type, see Sec. \ref{stran}) and the bowl soliton (a rotational surface). We refer to the reader the next references without to be a complete list: \cite{aw,css,ha,mar,mh,sh,sm,wa}.
\item A plane orthogonal to $\vec{v}$ is a $1/2$-translating soliton.
\item A circular cylinder of radius $r>0$ whose axis is parallel to $\vec{v}$ is a $1/(2r)$-translating soliton.
\end{enumerate}

In this paper we consider $\lambda$-translating solitons that are invariant by a one-parameter group of translations and a one-parameter group of rotations. In the first case, the group is characterized by the translation vector $\vec{a}$ and in the second one, by the rotation axis $\vec{a}$. We point out there is not an {\it a priori} relation between the direction $\vec{a}$ and the density vector $\vec{v}$ and a purpose of this paper is to study both types of invariant $\lambda$-translating solitons   in all its generality. We now give an approach to both  settings.

Invariant surfaces by a one-parameter group of translations are related with the one-dimensional problem of Eq. (\ref{eq1}) by  considering a planar curve $\gamma:I\rightarrow\r^2$,  $\gamma=\gamma(s)$, whose curvature $\kappa$ satisfies
\begin{equation}\label{e-one}
\kappa(s)=\lambda +\langle{\textbf n}(s),\vec{w}\rangle,\qquad s\in I,
\end{equation}
where $\vec{w}$ is a fixed vector of $\r^2$ and ${\textbf n}(s)$ is the principal unit normal vector of $\gamma$. Examples of  solutions of   (\ref{e-one}) are straight lines parallel to $\vec{w}$ ($\lambda=0$) and straight lines orthogonal to $\vec{w}$  ($\lambda=1$).  With a solution $\gamma$ of (\ref{e-one}), we can construct  a $\lambda$-translating soliton in $\r^3$ invariant by a group of translations as follows. If  $(x,y,z)$ stand for the usual coordinates of $\r^3$, we place  $\gamma$  in the $yz$-plane and consider the cylindrical surface  $\Sigma_\gamma=\gamma(I)\times \r e_1 $, where $e_1=(1,0,0)$. Then it is immediate that $\Sigma_\lambda$ is a $\lambda/2$-translating soliton with density vector $\vec{v}=(0,\vec{w})$ and $\Sigma_\lambda$ is invariant by the group of translations generated by $e_1$.

In general, a surface $\Sigma$  invariant by a one-parameter group of translations   can be parametrized as  $X(s,t)=\alpha(s)+t \vec{a}$, where $\alpha$ is a planar curve and $\vec{a}$ is a unit vector orthogonal to the plane containing $\gamma$ and it is called a \emph{cylindrical surface}. In the above example $\Sigma_\gamma$, the translation vector $\vec{a}$ is orthogonal to the density vector $\vec{v}$.   In Sec. \ref{scyl} we study all $\lambda$-translating solitons that are cylindrical surfaces, obtaining in Th. \ref{c1} a complete classification of these surfaces. This classification depends on the value of $\lambda$ and the vectors $\vec{a}$ and $\vec{v}$. We point out here that for a certain range of values of $\lambda$, there exist entire convex  surfaces. In Sec. \ref{stran} we extend the notion of cylindrical surface studying $\lambda$-translating solitons which are the sum of two planar curves of $\r^3$ contained in orthogonal planes and we classify these surfaces in Th. \ref{t31}.

In Sec. \ref{srot} we study $\lambda$-translating solitons invariant by a one-parameter group of rotations and    Ths. \ref{trot} and \ref{trot2}   give a complete classification of the  rotational $\lambda$-translating solitons. Firstly, we prove that the rotational axis must parallel to the density vector. Next, we prove in Th. \ref{tclosed} that there no exist closed $\lambda$-translating solitons, in particular,    there are not closed rotational examples. Among the examples that appear in our classification, we point out that for some range of $\lambda$, there exist  embedded surfaces that meet the rotational axis orthogonally which are  asymptotic to   right circular cylinders. We also  find convex entire graphs. Finally, in Sect. \ref{sriemann} we prove that if a   $\lambda$-translating soliton is  foliated by circles in parallel planes orthogonal to the density vector, then the surface must be a surface of revolution.

\section{ Cylindrical translating solitons}\label{scyl}

In this section we study $\lambda$-translating solitons invariant by a one-parameter group of translations $G=\{M_t:t\in\r\}$, where $M_t$ is the translation $M_t(p)=p+t\vec{a}$, $p\in\r^3$ and   $|\vec{a}|=1$.  A surface   $\Sigma\subset \r^3$ invariant by a such group is said   a \emph{cylindrical surface}. It follows from the definition that  a global parametrization of $\Sigma$ is $X(s,t)=\alpha(s)+t\vec{a}$, $s\in I$, $t\in\r$, where $\alpha:I\rightarrow\r^3$ is a curve whose trace is contained in a plane orthogonal to $\vec{a}$. The mean curvature of $\Sigma$ is $H(X(s,t))=\kappa(s)/2$, where $\kappa$ is the curvature of $\alpha$. Therefore Eq. (\ref{eq1}) is
\begin{equation}\label{hk}
 \kappa(s)=2\lambda+\mbox{det}(\alpha'(s),\vec{a},\vec{v}).
\end{equation}
 It is immediate that if $\alpha$ is a straight line with direction $\vec{w}$, then $\alpha$ satisfies (\ref{hk}) with $2\lambda=-\mbox{det}(\vec{w},\vec{a},\vec{v})$. Also, if $\vec{a}$ is parallel to $\vec{v}$, then (\ref{hk}) is equivalent to $\kappa=2\lambda$ is constant and thus, $\alpha$ is a straight line or $\alpha$ is a circle of radius $1/(2\lambda)$. We collect these cases:
\begin{proposition}\label{fc}
\begin{enumerate}
\item A   plane is a $\lambda$-translating soliton of cylindrical type for any density vector.
\item Planes and  right circular cylinders are the only $\lambda$-translating solitons of cylindrical type whose rulings are parallel to the density vector.
\end{enumerate}
\end{proposition}

Other particular case is   $\lambda=0$, that is, $\Sigma$ is a translating soliton. It is known that when $\vec{a}$ is orthogonal to $\vec{v}$, then the only cylindrical translating soliton is a   plane parallel to $\vec{a}$ and the   grim reaper  \cite{mh}. A parametrization of this surface for $\vec{a}=(1,0,0)$ and $\vec{v}=(0,0,1)$ is $X(s,t)=(t,s,-\log|\cos(s)|)$.

We now solve (\ref{hk}) in all its generality. Up to a change of coordinates, we take $\vec{a}=(1,0,0)$ and after a rotation about $\vec{a}$, we suppose  $\vec{v}=(v_1,0,v_3)$, with $v_3\geq 0$, $|\vec{v}|=1$. Then the parametrization  of $\Sigma$ is
$X(s,t)=(t,y(s),z(s))$, where $s\in I, t\in\r$ and  the curve $\alpha(s)=(0,y(s),z(s))$ is assumed to be parametrized by arc-length. Denote $\theta$ the angle that makes the velocity
$\alpha'(s)$ with the $y$-axis and let $y'(s)=\cos\theta(s)$ and
$z'(s)=\sin\theta(s)$ for a certain   function
$\theta$. The derivative $\theta'(s)$  is just
the  curvature $\kappa$ of $\alpha$. If $\Sigma$ is oriented with respect to the Gauss map $N(t,s)=(0,\sin\theta(s),-\cos\theta(s))$, then  \eqref{hk} is equivalent to
\begin{equation}\label{1}
 \left\{\begin{array}{lll}
 y'(s)&=& \displaystyle \cos\theta(s)\\
 z'(s)&=&\displaystyle \sin\theta(s)\\
 \theta'(s)&=&2\lambda+v_3\cos\theta(s).
\end{array}
\right.
\end{equation}

After a translation in the $yz$-plane we suppose that the initial conditions are
\begin{equation}\label{111}
y(0)=0,\ \ z(0)=0,\ \ \theta(0)=\theta_0,
\end{equation}
and denote $\{y(s;\lambda,\theta_0),z(s;\lambda,\theta_0),\theta(s;\lambda,\theta_0)\}$ the solution of \eqref{1}-\eqref{111}. It is immediate that the solutions   are defined in $\r$ because the  derivatives $y',z'$ and $\theta'$ are bounded. By Prop. \ref{fc}, we now assume that the function $\theta$ is not constant and that $v_3>0$. The solutions of (\ref{1}) satisfy the next symmetric properties:

\begin{proposition}\label{l1}  Let $\alpha(s)=(y(s),z(s))$ be a solution of (\ref{1}).
\begin{enumerate}
\item If the curvature of $\alpha$ vanishes at some point, then $\alpha$ is   a straight line.
\item If the tangent vector of $\alpha$ is horizontal at some point $s_0$, then the graphic of $\alpha$   is symmetric with respect to the vertical straight line $y=y(s_0)$.
\end{enumerate}
\end{proposition}

\begin{proof}
\begin{enumerate}
\item Suppose $s=0\in I$ is a point where the curvature of $\alpha$ vanishes, that is, $\theta'(0)=0$. Let $\theta_0=\theta(0)$. By the third equation in (\ref{1}),  we have $2\lambda=-v_3 \cos\theta_0$. The functions $\{\bar{y},\bar{z},\bar{\theta}\}$ defined as
\begin{eqnarray*}
&&\bar{y}(s)=(\cos\theta_0)s+y(0)\\
&&\bar{z}(s)=(\sin\theta_0)s+z(0)\\
&&\bar{\theta}(s)=\theta_0
\end{eqnarray*}
satisfy \eqref{1} with the same initial conditions at $s=0$ that $\{y,z,\theta\}$. By uniqueness of ODE, $\{y,z,\theta\}=\{\bar{y},\bar{z},\bar{\theta}\}$, proving the result.
 \item After a change in the parameter, suppose $s_0=0$. Since $\alpha'(0)$ is horizontal, then  up to an integer multiply of $2\pi$, we   have $\theta(0)=0$ or  $\theta(0)=\pi$.  Suppose $\theta(0)=0$ (similarly for  the other case). Then the functions $\{y,z,\theta\}$ satisfy \eqref{1} with initial conditions $(0,0,0)$. Define
\begin{eqnarray*}
\bar{y}(s)&=&-y(-s)+2y(0)\\
\bar{z}(s)&=& z(-s)\\
\bar{\theta}(s)&=&-\theta(-s).
\end{eqnarray*}
Then it is immediate that $\{\bar{y},\bar{z},\bar{\theta}\}$ satisfy \eqref{1} with the same initial conditions at $s=0$, and thus $\{\bar{y}(s),\bar{z}(s)\}=\{y(s),z(s)\}$, proving the result.
\end{enumerate}
\end{proof}

We relate the shape of a $\lambda$-translating cylindrical surface when we change the sign of $\lambda$. We point out that the computation of $\lambda$ in (\ref{1}) was obtained by fixing an orientation on $\Sigma$ and thus we have to consider any  value of $\lambda$. In the next result, we prove that  the graphics of a solution of  \eqref{1} for $\lambda$ and $-\lambda$ coincide up to  reparametrizations.

 \begin{proposition}\label{pr25} For suitable initial conditions in (\ref{111}), the graphic of a solution of (\ref{1}) for $\lambda$  coincides with the graphic of a solution of (\ref{1}) for $-\lambda$ and  we have
 $$(y(s;\lambda,\theta_0),z(s;\lambda,\theta_0)\}=\{-y(s;-\lambda,\pi-\theta_0), z(s;-\lambda,\pi-\theta_0)),\ s\in\r.$$
 \end{proposition}

 \begin{proof} The functions
 \begin{eqnarray*}
\bar{y}(s)&=&-y(s;-\lambda,\pi-\theta_0),\\
\bar{z}(s)&=& z(s;-\lambda,\pi-\theta_0),\\
\bar{\theta}(s)&=&\pi-\theta(s;-\lambda,\pi-\theta_0),
\end{eqnarray*}
 satisfy \eqref{1} with $(0,0,\theta_0)$ as initial conditions at $s=0$, proving  the result.
 \end{proof}

The next result classifies all non-planar $\lambda$-translating solitons that are cylindrical surfaces whose density vector is non-parallel to the rulings.

 \begin{theorem}\label{c1}  Let $\alpha(s)=(y(s;\lambda),z(s;\lambda))$ be a solution of (\ref{1})-(\ref{111}) which is not a straight line. After a change of the initial conditions in (\ref{111}), and by Prop. \ref{pr25}, we   assume    $\lambda\geq 0$ and $\vec{v}=(v_1,v_2,v_3)$, $v_3>0$. Then we have:
  \begin{enumerate}
 \item Case $\lambda>v_3/2$. The graphic of $\alpha$    is invariant by a discrete group of horizontal translations in the $yz$-plane and the angle function $\theta$ rotates infinitely times around the origin. See Fig. \ref{f1}, left.
 \item Case $\lambda=v_3/2$.  The graphic of $\alpha$ has one self-intersection point and it is symmetric about a vertical line. See Fig. \ref{f1}, right.
 \item Case $0<\lambda<v_3/2$.  The graphic of $\alpha$ has two branches asymptotic to two lines  of slopes $\pm\tan\theta_1$, where $2\lambda+\cos\theta_1=0$. Depending on the initial values, we have:
 \begin{enumerate}
 \item The curve $\alpha$ is symmetric about a vertical line, it has one point of self intersection. See Fig. \ref{f2}, left.
 \item The curve $\alpha$ is a convex graph on the $y$-real line. See Fig. \ref{f2}, middle.

 \end{enumerate}

 \item Case $\lambda=0$. The solution is  the grim reaper. See Fig. \ref{f2}, right.
 \end{enumerate}
\end{theorem}
\begin{proof}

Since $\alpha$ is not a straight line, Prop. \ref{l1} asserts that the derivative of $\theta$ cannot vanish, and thus $\theta=\theta(s)$ is a monotone function. The monotonicity of $\theta$ is given by the value $\theta'(0)=2\lambda+v_3 \cos\theta_0$.
\begin{enumerate}
 \item Case $\lambda>v_3/2$.  Then $\theta'=2\lambda+v_3\cos\theta>1$ and thus $\theta$ is increasing and $\lim_{s\rightarrow\pm\infty}\theta(s)=\pm\infty$, which   proves the second part of the statement. Moreover, $\theta$ takes all real values and it suffices to assume $\theta_0=0$ by Prop. \ref{pr25}. Let $T>0$ be the unique number such that $\theta(T)=2\pi$ and denote $y_0=y(T)$. The uniqueness of ODE. gives immediately
$$(y(s+T;\lambda,0),  z(s+T;\lambda,0))=(y(s;\lambda,0)+y(T),z(s;\lambda,0)),$$
with $\theta(s+T)=\theta(s)+2\pi$.  This proves that the graphic of $\alpha$ is invariant by the group of translation of $\r^2$ generated by the  vector $(y(T),0)$.
\item Case $\lambda=v_3/2$. As $\theta'=v_3(1+\cos\theta)$, then $\theta$ is a monotone increasing function  that can not attain the values $\pm\pi$.  After a change in the parameter $s$, we suppose $\theta(0)=0$, $s=0$ is the only point where $\theta$ vanishes and $\lim_{s\rightarrow\pm\infty}\theta(s)=\pm \pi$.  Thus the graphic of $\alpha$ is horizontal at $s=0$ and Prop. \ref{l1} applies.  As there exist two values $s$ such that $y'(s)=\cos\theta(s)=0$, then $\alpha$ is not a graph on the $y$-axis. Since $z'(0)=0$ and $z''(0)=\theta'(0)=2v_3>0$, the graphic is symmetric about the line $y=0$ with a minimum at $s=0$. Finally, for $s>0$,
$$y(s)=\int_0^s\cos\theta(t)dt=\int_0^s\left(\frac{\theta'(t)}{v_3}-1\right)dt=\frac{\theta(s)}{v_3}-s.$$
By symmetry, we have $\lim_{s\rightarrow\pm\infty}y(s)=\mp\infty$ and since $\theta$ moves from the value $-\pi$ to $\pi$, the graphic of $\alpha$ has a   point self intersection.
\item Case $0<\lambda<v_3/2$.  Since $\theta'=2\lambda+v_3\cos\theta\neq 0$ for all $s\in\r$, the range of $\theta$ is the union of two open intervals $J_1\cup J_2$, namely,
$$J_1=(-\cos^{-1}(-2\lambda/v_3),\cos^{-1}(-2\lambda/v_3)),$$
$$J_2=\left(\cos^{-1}(-2\lambda/v_3),2\pi-\cos^{-1}(-2\lambda/v_3)\right).$$
The behavior of a solution of (\ref{1}) depends if the value $\theta_0$ in (\ref{111})  lies in $J_1$ or in $J_2$.
\begin{enumerate}
\item Case  $\theta_0\in J_1$. Without loss of generality, we suppose $\theta_0=0$. In particular, the graphic of $\alpha$ is symmetric about the line $y=0$ and,  similar as in the case (2), the graphic of $\alpha$ has a minimum at $s=0$ with a  point of  self intersection.
\item Case  $\theta_0\in J_2$. Without loss of generality, we suppose $\theta_0=\pi$. Then  $\theta'<0$, that is, $\theta$ is an decreasing function. As the range of $\theta$  is   $J_2$, then    $y'(s)\not=0$ for any $s$ and this means that $\alpha$ is a graph   on the $y$-line and $z=z(y)$. Moreover,    $z''(0)=\theta'(0)\cos\theta(0)>0$, that is, $z$ attains a minimum. By symmetry, $\alpha$ is   a graph on the $y$ axis,   symmetric about   the line $y=0$ and with a minimum at $(0,0)$. Finally,
$$z''(y)=\frac{1}{\cos^3\theta(s)}\theta'(s)>0,$$
because $\theta'$ and $\cos\theta$ are both negative and thus the graph of $\alpha$ is  convex.
\end{enumerate}
If $\theta_1=\cos^{-1}(-2\lambda/v_3)$ and since $\theta$ is monotone, the graphic of $\alpha$ is asymptotic to two lines of slopes $\pm\tan\theta_1$.

\item Case $\lambda=0$. The result is known (\cite{mh}).
\end{enumerate}
\end{proof}

\begin{figure}[h]
\begin{center}
\includegraphics[width=.4\textwidth]{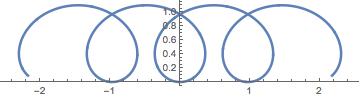}\ \includegraphics[width=.4\textwidth]{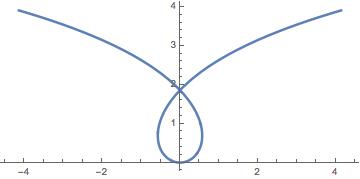}
\end{center}
\caption{Invariant $\lambda$-translating solitons for $\lambda=1$ (left) and $\lambda=1/2$ (right). Here $\vec{v}=(0,0,1)$ and the initial condition is  $\theta(0)=0$}\label{f1}
\end{figure}

\begin{figure}[h]
\begin{center}
\includegraphics[width=.14\textwidth]{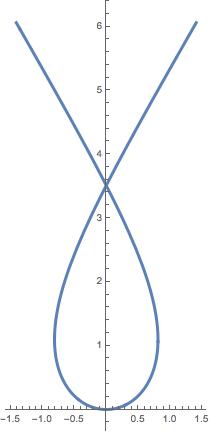}\qquad \includegraphics[width=.3\textwidth]{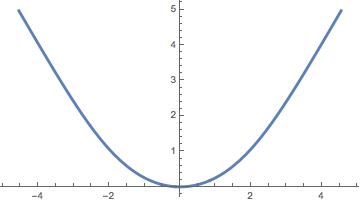}  \qquad \includegraphics[width=.23\textwidth]{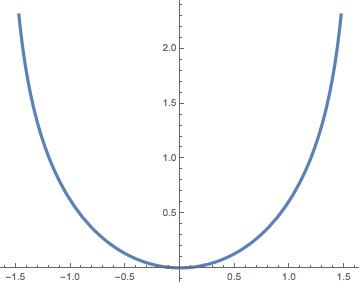}
\end{center}
\caption{Left and middle: invariant $\lambda$-translating solitons for $\lambda=1/4$ and initial condition $\theta_0=0$ and $\theta_0=\pi$, respectively. Right:  the grim reaper.  Here $\vec{v}=(0,0,1)$}\label{f2}
\end{figure}
\begin{remark} By Prop. \ref{fc} and   Th. \ref{c1}, the only cases where a $\lambda$-translating soliton of cylindrical type is a graph on a planar domain are the grim reaper ($\lambda=0$), which is  a graph on a strip, and the case $\lambda\in (0,v_3/2)$,  where the surface is an entire graph. In both cases, the graph is  convex.
\end{remark}

We finish this section   indicating that it is possible to integrate  explicitly (\ref{1}). The third equation in (\ref{1})  is
$$\int\frac{d\theta}{2\lambda+v_3\cos\theta}=s+a,\ a\in\r.$$
Then  the change $u=\tan(\theta/2)$ gives, up to a linear change in the parameter $s$,
$$\theta(s)=\left\{\begin{array}{ll}
2\arctan\left(\sqrt{\frac{2\lambda+v_3}{2\lambda-v_3}}\tan\left(\frac{\sqrt{4\lambda^2-v_3^2}}{2}s\right)\right)& \lambda>\frac{v_3}{2}\\
2\arctan(v_3s)&\lambda=\frac{v_3}{2}\\
2\arctan\left(\sqrt{\frac{2\lambda+v_3}{-2\lambda+v_3}}\tanh\left(\frac{\sqrt{v_3^2-4\lambda^2}}{2}s\right)\right)& 0<\lambda<\frac{v_3}{2}
\end{array}\right.$$
Once obtained the function $\theta$, we compute $\cos\theta$ and $\sin\theta$ in order to solve the functions $y$ and $z$ in (\ref{1}). This can be only done in those intervals where is defined the  $\arctan$ function. The explicit integration gives:

\begin{theorem}\label{t2} The generating curve $\alpha(s)=(y(s),z(s))$ of a cylindrical $\lambda$-translating soliton $X(s,t)=(t,y(s),z(s))$, $s\in I$, $t\in{\mathbb R}$ is:
\begin{enumerate}
\item Case $\lambda>v_3/2$.
\begin{eqnarray*}
&&y(s)=-2\lambda s+2\arctan\left(\frac{2\lambda+v_3}{\sqrt{4\lambda^2-v_3}}\tan(\frac{\sqrt{4\lambda^2-v_3^2}}{2}s)\right)\\
&&z(s)= \log\left|2\lambda-\cos(\sqrt{4\lambda^2-v_3^2}s)\right|.
\end{eqnarray*}
\item Case $\lambda=v_3/2$.
$$\alpha(s)=\left(-s+\frac{2}{v_3}\arctan(v_3s), \frac{1}{v_3}\log|1+s^2v_3^2|\right).$$
\item Case $ 0<\lambda< v_3/2$.
\begin{eqnarray*}
&&y(s)=-2\lambda s+2\arctan\left(\frac{2\lambda+v_3}{\sqrt{v_3^2-4\lambda^2}}\tanh(\frac{\sqrt{v_3^2-4\lambda^2}}{2}s)\right)\\
&&z(s)=\log\left|-2\lambda+\cosh(\sqrt{v_3^2-4\lambda^2}s)\right|.
\end{eqnarray*}

\item Case $\lambda=0$ (grim reaper).
$$\alpha(s)=\left(2\arctan(\tanh(\frac{s}{2})),\log(\cosh(s))\right).$$
\end{enumerate}
\end{theorem}

\section{ Translating solitons of translation type}\label{stran}

The parametrization  $X(s,t)=\alpha(s)+t\vec{a}$ of  a cylindrical surface  allows to see the surface   as the sum of two planar curves, namely,   $X(s,t)=\alpha(s)+\beta(t)$, where $\beta$ is the straight line $\beta(t)=t\vec{a}$. More generally, we can consider the solutions of (\ref{eq1}) that are the sum of two planar curves $\alpha(x)=(x,0,f(x))$ and $\beta(y)=(0,y,g(y))$. The parametrization $X(x,y)=(x,y,f(x)+g(y))$ is noting that the surface in the non-parametric form $z=f(x)+g(y)$. A such surface is called a \emph{translation surface}  and in this section we investigate the solutions of (\ref{eq1}) that are translation surfaces. For translating solitons $(\lambda=0)$ and when the density vector is orthogonal to the $xy$-plane,   it is  known that the only translating solitons of translation type are cylindrical surfaces (\cite{mh}). Exactly, and besides the plane, the functions $f$ and $g$ are,    up to a change of the roles of $f$ and  $g$,    $f(x)=ax+b$  and
$$g(y)=-(1+a^2)\log\left|\cos\left(\frac{x}{\sqrt{1+a^2}}\right)\right|,$$
where $a,b\in\r$. The surface for $a=0$ is the grim reaper. When $\vec{v}$ is not orthogonal to the $xy$-plane, there exist many examples of translating solitons  $z=f(x)+g(y)$ (\cite{lo2}).

In this section, we find  all  $\lambda$-translating solitons of translation type when   the density vector $\vec{v}$ takes all its generality.

 \begin{theorem}\label{t31} Let $\vec{v}=(v_1,v_2,v_3)$ be the density vector. The only $\lambda$-translating solitons  $z=f(x)+g(y)$ are:
 \begin{enumerate}
 \item Planes and  $f$ and $g$ are linear functions.
 \item  Up to a change of the roles of $f$ and $g$, we have $f(x)=ax+b$, and $g$ satisfies
 \begin{equation}\label{gg}
(1+a^2)g''=2\lambda (1+a^2+g'^2)^{3/2}+(1+a^2+g'^2)(-v_1 a-v_2 g'+v_3),\end{equation}
where $a,b\in\r$. In particular, the surface is cylindrical  and  the rulings are parallel to the vector $(1,0,a)$.
\end{enumerate}
 \end{theorem}

\begin{proof} The case $\lambda=0$ was studied in \cite{lo2,mh}.  Consider now $\lambda\not=0$.   Using the parametrization  $X(x,y)=(x,y,f(x)+g(y))$, with $(x,y)\in I\times J\subset\r^2$, the Gauss map is
$$\langle N,\vec{v}\rangle=\frac{-v_1 f'-v_2 g'+v_3}{W^{1/2}},$$
 where $W=1+f'^2+g'^2$. Then Eq. \eqref{eq1} is
\begin{equation}\label{tr1}
\frac{(1+g'^2)f''+(1+f'^2)g''}{W^{3/2}}=2\lambda+\frac{-v_1 f'-v_2 g'+v_3}{W^{1/2}}.
\end{equation}
Multiplying by $W^{3/2}$ and  differentiating with respect to $x$,  next with respect to $y$ and simplifying, we get
\begin{equation}\label{tr2}
g'g''f'''+f'f''g'''=3\lambda\frac{f'f''g'g''}{W^{1/2}}-(v_1 f''g'g''+v_2 f'f''g'').
\end{equation}
Suppose that at some $(x_0,y_0)\in I\times J$ we have $f'f''(x_0)g'g''(y_0)\not=0$. By continuity, in some open set around $(x_0,y_0)$, we have $f'f''g'g''\not=0$. Dividing (\ref{tr2}) by $f'f''g'g''$, we obtain
$$\frac{f'''}{f'f''}+\frac{g'''}{g'g''}+\frac{v_1}{f'}+\frac{v_2}{g'}=\frac{3\lambda}{W^{1/2}}.$$
Since the left hand side is the sum of a function on $x$ and a function on $y$, when we differentiate with respect to $x$ and next with respect to $y$, the left hand side vanishes. Doing the same differentiations in the right hand side, we get
$$0=9\lambda\frac{f'f''g'g''}{W^{5/2}},$$
obtaining a contradiction because $\lambda\not=0$. The above argument proves $f'f''g'g''=0$ in $I\times J$. Without loss of generality, and by the symmetry of the roles of $f$ and $g$, we suppose $f'f''=0$ in the interval $I$. If at some point $x_0\in I$, we have $f'(x_0)\not=0$, then $f''=0$ around $x_0$, that is,  $f(x)=ax+b$, $a,b\in\r$. With this function $f$,  Eq. \eqref{tr1} reduces into  \eqref{gg}, obtaining the result. In the case that $f'=0$, then $f''=0$, concluding the same result.
\end{proof}

We finish this section showing two examples of $\lambda$-translating solitons of translation type with $\lambda\not=0$.
\begin{enumerate}
\item Consider  the density vector to be   $\vec{v}=(0,0,1)$. Take  $f(x)=0$ and $\lambda=1$ in (\ref{gg}). Then $g$ satisfies $g''=1+g'^2+2(1+g'^2)^{3/2}$. This surface appears in Th. \ref{c1}, item 1

\item Consider   $\vec{v}=(1,0,0)$ as the density vector. Take $f(x)=0$. Then Eq. (\ref{gg}) is $g''=2\lambda (1+g'^2)^{3/2}$ whose solution is
$$g(y)=\frac{\sqrt{1-4\lambda^2y^2}}{2\lambda},$$
that is, $y=g(y)$ describes a circle of radius $1/(2|\lambda|)$ and the surface is a circular right cylinder whose axis is $\vec{v}$. This solution appeared in  Prop. \ref{fc}.
\end{enumerate}

\section{ Rotational $\lambda$-translating solitons}\label{srot}

In this section, we classify all  rotational surfaces that are $\lambda$-translating solitons. First examples are a plane orthogonal to $\vec{v}$ and a right circular cylinder with axis parallel to $\vec{v}$. In the particular case $\lambda=0$,   the translating solitons of rotational type were studied in \cite{css}, obtaining two types of surfaces, namely, the paraboloid bowl soliton and a family of  rotationally   surfaces of winglike shape.

Our interest is also those solutions with a particular geometry as for example, when the surface meets  the rotation axis or if it is embedded. A  first question is about the existence of closed surfaces.  Let us recall that  the round sphere is the only rotational constant mean curvature    surface that is closed and that   there are many  examples of closed   surfaces with constant mean curvature which are not rotational.   However, for $\lambda$-translating solitons we have:

\begin{theorem} \label{tclosed}
There are no closed $\lambda$-translating solitons.
\end{theorem}

\begin{proof} By contradiction, let $\psi:\Sigma\rightarrow\r^{3}$ be an immersion of a closed surface $\Sigma$ whose mean curvature $H$ satisfies (\ref{eq1}). It is known that if $\vec{a}\in\r^{3}$,   the Laplacian $\Delta$ of the height function $\langle \psi,\vec{a}\rangle$ is $\Delta\langle\psi,\vec{a}\rangle=2H\langle N,\vec{a}\rangle$. If  we take $\vec{a}=\vec{v}$, we have
\begin{equation}\label{lapla}
\Delta\langle\psi,\vec{v}\rangle=2\lambda\langle N,\vec{v}\rangle +\langle N,\vec{v}\rangle^2.
\end{equation}
 We integrate this identity in $\Sigma$. By using the divergence theorem and because $\partial\Sigma=\emptyset$, we have
\begin{equation}\label{com}
0=2\lambda\int_\Sigma \langle N,\vec{v}\rangle\ d\Sigma+\int_\Sigma\langle N,\vec{v}\rangle^2\ d\Sigma.
\end{equation}
On the other hand,  the constant vector field in $\r^3$ defined by $Y(p)=\vec{v}$ has zero divergence and thus the divergence theorem gives now $\int_\Sigma\langle N,\vec{v}\rangle\ d\Sigma=0$. We conclude from \eqref{com} that $0=\int_\Sigma\langle N,\vec{v}\rangle^2\ d\Sigma$, that is, $\Sigma$ is included in a plane parallel to $\vec{v}$, a contradiction.
\end{proof}

\begin{remark} In the literature, the proof of Th. \ref{tclosed} for translating solitons ($\lambda=0$) uses  the maximum principle for (\ref{eq1}) and an argument of comparison with planes parallel to $\vec{v}$. However, this proof fails if $\lambda\not=0$. In contrast, the proof given in Th. \ref{tclosed}  is simpler because only uses the divergence theorem and it holds for any $\lambda$.
\end{remark}

Although in our initial study there is not an {\it a priori} relation  between the rotational axis and the density vector, we prove that they must be parallel.

\begin{proposition}\label{pr1} Let $\Sigma$ be  a rotational surface about the axis $L$. If $\Sigma$ is a $\lambda$-translating soliton with density vector $\vec{v}$, then $\vec{v}$ and $L$ are parallel or $\Sigma$ is a plane orthogonal to $L$ and $\vec{v}$.
\end{proposition}
\begin{proof}  After a change of coordinates, we suppose that the rotational axis is the $z$-axis. A parametrization of $\Sigma$ is  $X(s, t)=(x(s)\cos t,x(s)\sin t,z(s))$, $s\in I$, $ t\in\r$, where $\alpha(s)= (x(s),0,z(s))$, $x(s)>0$, is the profile curve which we suppose is parametrized by the arc-length. Let $x'(s)=\cos\theta(s)$, $z'(s)=\sin\theta(s)$ for some function $\theta$. If $\vec{v}=(v_1,v_2,v_3)$, then Eq. \eqref{eq1} is
$$\frac{\sin\theta}{x}+\theta'=2\lambda-v_1\sin\theta\cos t-v_2\sin\theta\sin t+v_3\cos\theta,$$
for all $s\in I, t\in\r$. Since the functions $\{\cos t,\sin t,1\}$ are independent linearly, we deduce $$v_1\sin\theta(s)=v_2\sin\theta(s)=0$$
 for all $s\in I$. If there exists $s_0\in I$ such that $\sin\theta(s_0)\not=0$, then $v_1=v_2=0$ and $\vec{v}$ is parallel to the $z$-axis, proving the result. On the contrary, the function  $\sin\theta$ is identically $0$ in the interval $I$.   This means that  $\alpha$ is a horizontal line and $\Sigma$ is a horizontal plane: now  the density vector $\vec{v}$ is  arbitrary.
\end{proof}

As a consequence of  Prop. \ref{pr1} and  without loss of generality, we suppose that the rotational axis is the $z$-axis and $\vec{v}=(0,0,1)$. As before, if $\alpha(s)=(x(s),0,z(s))$  is the profile curve, then the functions $x$, $z$ and $\theta$ satisfy
\begin{equation}\label{11}
 \left\{\begin{array}{lll}
 x'(s)&=&  \cos\theta(s)\\
 z'(s)&=& \sin\theta(s)\\
 \theta'(s)&=&\displaystyle2\lambda+\cos\theta(s)-\frac{\sin\theta(s)}{x(s)}.
\end{array}
\right.
\end{equation}

A particular case of (\ref{11}) appears when  $\theta$ is a constant function. Then $\alpha$ is a straight line and $2\lambda+\cos\theta+\sin\theta/x(s)=0$. From the first equation in (\ref{11}), we know   $x(s)=a+(\cos\theta) s$, $a\in\r$, and substituting in the third equation of (\ref{11}), we have
$$\cos\theta(2\lambda+\cos\theta)s+2\lambda a+a\cos\theta-\sin\theta=0.$$
Since this is a polynomial equation on the variable $s$, we deduce
\begin{equation}\label{cc}
\cos\theta(2\lambda +\cos\theta)=0,\quad 2\lambda a+a\cos\theta-\sin\theta=0.
\end{equation}
If $\cos\theta=0$, then   $\alpha$ is the vertical line of equation $x=1/(2|\lambda|)$ and the surface is a right circular cylinder. If $\cos\theta\not=0$, we have from \eqref{cc} that $\sin\theta=0$ and thus $\alpha$ is a horizontal line and $\Sigma$ is a horizontal plane. Therefore we have proved the next result:

\begin{proposition} The only rotational $\lambda$-translating solitons generated by straight lines are planes   and  right circular cylinders of radius $1/(2|\lambda|)$.
\end{proposition}

We give now the relationship between  the shape of a rotational $\lambda$-translating soliton  and  the sign of $\lambda$.  After a vertical translation, we take the initial conditions
\begin{equation}\label{112}
x(0)=x_0>0,\ \ z(0)=0,\ \ \theta(0)=\theta_0,
\end{equation}
and denote $x(s;\lambda,\theta_0),\ z(s;\lambda,\theta_0),\ \theta(s;\lambda,\theta_0)$ the solutions of (\ref{11}) with initial conditions (\ref{112}) depending on $\theta_0$. The next result is analogous to Prop. \ref{pr25}.

 \begin{proposition}\label{pr26} For suitable initial conditions in (\ref{112}), the graphic of a solution  of (\ref{11}) for $\lambda$  coincides with the graphic of a solution of (\ref{11}) for $-\lambda$ and  we have
 $$(x(s;\lambda,\theta_0),z(s;\lambda,\theta_0))=(x(-s;-\lambda,\theta_0+\pi), z(-s;-\lambda,\theta_0+\pi)),\ s\in\r.$$
 \end{proposition}

 \begin{proof} The proof is analogous to Prop. \ref{pr25} by defining the functions
 \begin{eqnarray*}
\bar{x}(s)&=&x(-s;-\lambda,\theta_0+\pi)\\
\bar{z}(s)&=& z(-s;-\lambda,\theta_0+\pi)\\
\bar{\theta}(s)&=&\theta(-s;-\lambda,\theta_0+\pi)+\pi.
\end{eqnarray*}
 \end{proof}


We come back to the ODE system (\ref{11}). Multiplying the third equation by $x(s)$, we have
$x\theta'=2\lambda x+xx'-z'$ and thus
$$(x\sin\theta)'=x'z'+\cos\theta x\theta'=x'z'+\cos\theta(2\lambda x+xx'-z')=(\lambda x^2)'+xx'^2.$$
If we fix $s_0\in I$, then
\begin{equation}\label{int}
x(s)\sin\theta(s)-\lambda x(s)^2=x(s_0)\sin\theta(s_0)-\lambda x(s_0)^2+\int_{s_0}^sx(t)x'(t)^2 dt.
\end{equation}
We prove that if the graphic of $\alpha$ meets the rotational axis, then this intersection is orthogonal.

\begin{proposition}\label{pr46}
 If the profile curve of a rotational $\lambda$-translating soliton intersects the rotational axis, then it does so at a perpendicular angle.
\end{proposition}

\begin{proof} Without loss of generality, suppose that the intersection between the curve and the axis occurs at $s=0$. Then  $x(0)=0$ and from (\ref{int}), we have
\begin{equation}\label{xx}
x(s)\sin\theta(s)-\lambda x(s)^2=\int_{0}^sx(t)x'(t)^2 dt.
\end{equation}
We divide this expression by $x(s)$, obtaining
$$\sin\theta(s)-\lambda x(s)=\frac{1}{x(s)}\int_{0}^sx(t)x'(t)^2 dt.$$
Letting $s\rightarrow 0$ and applying the L'H\^{o}pital rule, we obtain
$$
\sin\theta(0)=\frac{x(0) x'(0)^2}{x'(0)}=x(0)x'(0)=0,
$$
and this proves the result.
\end{proof}

We study the existence of solutions of (\ref{11})-(\ref{112}). The local existence is assured if $x(0)>0$. When $x(0)=0$, the third equation in (\ref{11}) presents a singularity and thus the existence is not a direct consequence of the  standard theory. We study this case. By Props. \ref{pr26} and \ref{pr46}, the initial condition for $\theta$ is $\theta(0)=0$. We give a proof of the existence using known techniques of the theory of the radial solutions for an elliptic equation. Here we prefer to write (\ref{eq1}) (or (\ref{11})) as the prescribed mean curvature equation
\begin{equation}\label{rot1}
\frac{u''(r)}{(1+u'(r)^2)^{3/2}}+\frac{u'(r)}{r\sqrt{1+u'(r)^2}}=2H=2\lambda+\frac{1}{\sqrt{1+u'(r)^2}},
\end{equation}
where, as usually, $r$ is the radial variable and $u=u(r)$.
 Multiplying   (\ref{rot1}) by $r$, we want to establish the existence of a classical solution of
\begin{equation}\label{rot-r}
\left\{\begin{array}{ll}
 \left({\displaystyle \frac{r u'(r)}{\sqrt{1+u'(r)^2}}}\right)'=2\lambda r+{\displaystyle \frac{r}{\sqrt{1+u'(r)^2}}},&\mbox{ in $(r_0,r_0+\delta)$}\\
u(r_0)=0, \quad u'(r_0)=0&
\end{array}\right.
\end{equation}
where $r_0\geq 0$. We consider the case $r_0=0$, where    Eq.  (\ref{rot-r}) is degenerate.

\begin{proposition}\label{pr-exi}
The initial value problem (\ref{rot-r}) with $r_0=0$ has a solution $u\in C^2([0,R])$ for some $R>0$ which depends continuously on the initial datum.
\end{proposition}

\begin{proof}
Define the functions $g:\r_0^+\times\r\rightarrow\r$ and $\varphi:\r\rightarrow\r$ by
$$g(x,y)=2\lambda+\frac{1}{\sqrt{1+y^2}},\ \varphi(y)=\frac{y}{\sqrt{1+y^2}}.$$
It is clear that a function $u\in C^2([0,\delta])$, for some $\delta>0$, is a solution of (\ref{rot-r}) if and only if   $r g(u,u')=(r\varphi(u'))'$ and $u(0)=0$, $u'(0)=0$.

Fix $\delta>0$ to be determined later and define the operator ${\mathcal S}$ by
$$({\mathcal S}u)(r)=a+\int_0^r\varphi^{-1}\left(\int_0^s\frac{t}{s}\left(2\lambda+\frac{1}{\sqrt{1+u'^2}}\right)dt\right)ds.$$
Then a fixed point of the operator ${\mathcal S}$ is a solution of the initial value problem (\ref{rot-r}). We prove that ${\mathcal S}$ is a contraction in the space $C^1([0,\delta])$ endowed the usual norm $\|u\|=\|u\|_\infty+\|u'\|_\infty$. The functions $g$ and  $\varphi^{-1}$ are  Lipschitz continuous of constant $L>0$ in
$[-\epsilon,\epsilon]\times[-\epsilon,\epsilon]$ and $[-\epsilon,\epsilon]$, respectively provided $\epsilon>0$ and $\epsilon<1$. Then for all $u,v\in\overline{B(0,\epsilon)}$ and for all $r\in [0,\delta]$,
$$|({\mathcal S}u)(r)-({\mathcal S}v)(r)|\leq\frac{L^2}{4} r^2\left( \|u-v\|_\infty+\|u'-v'\|_\infty\right)$$
$$|({\mathcal S}u)'(r)-({\mathcal S}v)'(r)|\leq\frac{L^2}{2} r\left( \|u-v\|_\infty+\|u'-v'\|_\infty\right)$$
Hence choosing $\delta>0$ small enough, we conclude that ${\mathcal S}$ is a contraction in the closed ball $\overline{B(0,\delta)}$ in $C^1([0,\delta])$. Thus the Schauder Point Fixed theorem proves the existence of  a local solution of the initial value problem  (\ref{rot-r}). This solution lies in $C^1([0,\delta])\cap C^2(0,\delta])$ and the $C^2$-regularity up to $0$ is verified directly by  using the L'H\^{o}pital rule: from (\ref{rot1})  we have      $$u''(0)+\lim_{r\rightarrow 0}\frac{u(r)}{r}=2\lambda+1,$$
that is,
$$\lim_{r\rightarrow 0} u''(r)=\lambda+\frac12.$$
The continuous dependence of local solutions on the initial datum  is a consequence of the continuous dependence of the fixed points of ${\mathcal S}$.
 \end{proof}

In a first step of the classification of the rotational $\lambda$-translating solitons, we study   the   solutions  of (\ref{11}) that intersect orthogonally the rotational axis. This means   $x_0=\theta_0=0$ in (\ref{112}).  We write here the third equation of (\ref{11}), namely,
\begin{equation}\label{3}
\theta'(s)=2\lambda+\cos\theta(s)-\frac{\sin\theta(s)}{x(s)}.
\end{equation}
As first observations, we have:
\begin{enumerate}
\item The monotonicity of the angle function $\theta$ close to $s=0$ is given by the value $\theta'(0)$. Equation (\ref{3}) and the L'H\^{o}pital rule gives $\theta'(0)=\lambda+1/2$. Therefore $\theta$ is increasing (resp. decreasing) around $s=0$ if $\lambda>-1/2$ (resp. $\lambda<-1/2$).
\item  If $\lambda>0$, from  (\ref{xx}) we deduce that   $\theta$ does not attain the value $  \pi$. Similarly, the function $\theta$ does not attain again the value $0$ because if $s_1>0$ is the first time where $\theta(s_1)=0$, then $\theta'(s_1)\leq 0$, but (\ref{11}) gives $\theta'(s_1)=2\lambda+1>0$. This contradiction proves  that $\theta$ is a bounded function $0<\theta(s)<\pi$. In particular, the solutions of (\ref{11}) are defined in $(0,\infty)$. Moreover, from (\ref{xx}) again, $x(s)-\lambda x(s)^2>0$ for every $s$ and we deduce that the function $x(s)$ is bounded, namely, $x(s)<1/\lambda$.
\item If $\lambda=0$, then the surface is the bow soliton (\cite{css}).
\item If $\lambda=-1/2$, then it is immediate that the solution of (\ref{11})-(\ref{112}) for $x_0=\theta_0=0$ is $x(s)=s$, $z(s)=0$ and $\theta(s)=0$, that is, $\alpha$ is a horizontal line and the surface is a horizontal plane.
\end{enumerate}
From now we discard the case $\lambda=0$. In order to give a description of the profiles curves,  we do an analytic study of the solutions of (\ref{11})  from the viewpoint of the dynamic system theory. Here we follow a similar study done by Gomes in   \cite{go2} in the classification of the rotational surfaces of spherical type with constant mean curvature in hyperbolic space ${\mathbb H}^3$ (see also \cite{go1} for other types of rotational surfaces in ${\mathbb H}^3$ and  the Euclidean case).  We project   the vector field $(x'(s),z'(s),\theta'(s))$  on the $(\theta,x)$  plane, obtaining the one-parameter plane vector field
$$
 \left\{\begin{array}{lll}
 \theta'(s)&=&\displaystyle2\lambda+\cos\theta(s)-\frac{\sin\theta(s)}{x(s)},\\
  x'(s)&=&  \cos\theta(s).
\end{array}
\right.
$$
Multiplying the vector field $(\theta,x)$  by $x$, which is positive, to eliminate the poles, we conclude that the above system is equivalent to the next autonomous system
\begin{equation}\label{pp}
 \left\{\begin{array}{lll}
 \theta'(s)&=&\displaystyle2\lambda x(s)+ x(s) \cos\theta(s) - \sin\theta(s)\\
 x'(s)&=&  x(s)\cos\theta(s).
\end{array}
\right.
\end{equation}
We study the qualitative properties of the solutions of (\ref{pp}).  By the periodicity of the functions $\sin\theta$ and $\cos\theta$, it suffices to consider  $\theta\in [-\pi,\pi]$. In the region $[-\pi,\pi]\times\{(\theta,x):x\geq 0\}$, the singularities of the vector field $V(\theta,x)=(  2\lambda x+ x \cos\theta - \sin\theta,x\cos\theta)$ are the points $P_1=(0,0)$, $P_2=(\pi,0)$, $P_3=(-\pi,0)$ and, furthermore, the point $Q_1=(\pi/2,1/(2\lambda))$ in case $\lambda>0$, and  the point $Q_2=(-\pi/2,-1/(2\lambda))$ if $\lambda<0$. We study the type of critical point in all these cases. The linearization of $V$ is
$$LV(\theta,x) =\left(\begin{array}{cc}-x\sin\theta-\cos\theta&2\lambda+\cos\theta\\ -x\sin\theta&\cos\theta\end{array}\right),$$
and denote $\mu_1$ and $\mu_2$ the two eigenvalues. The critical points   $P_i$ are hyperbolic  because the eigenvalues of $LV(P_i)$ are real with $\mu_1<0<\mu_2$. For the points $Q_i$, the eigenvalues $\mu_1$ and $\mu_2$  are $(-1\pm\sqrt{1-16\lambda^2})/(4\lambda)$ and the types of singularities appear in Tables \ref{tab1} and \ref{tab2}.
\begin{table}[h]
    \centering
        \begin{tabular}{| l |  l | l| l|}
        \hline
        $Q_1$& $0<\lambda<1/4$  &$\lambda=1/4$ &$\lambda>1/4$ \\
        \hline
        eigenvalues & $\mu_1<\mu_2<0$ & $\mu_1=\mu_2<0$ &$\Re(\mu_i)<0$ \\
         \hline
        type & stable node & stable improper node  & stable spiral point \\

        \hline

        \end{tabular}
    \caption{Types of singularity at $Q_1$}\label{tab1}
\end{table}

\begin{table}[h]
    \centering
        \begin{tabular}{| l |  l | l| l|}
        \hline
        $Q_2$& $-1/4<\lambda<0$  &$\lambda=-1/4$ &$\lambda<-1/4$ \\
        \hline
        eigenvalues & $\mu_1>\mu_2>0$ & $\mu_1=\mu_2>0$ &$\Re(\mu_i)>0$ \\
         \hline
        type & unstable node & unstable  improper node & unstable spiral point \\

        \hline

        \end{tabular}
    \caption{Types of singularity at $Q_2$}\label{tab2}
\end{table}

 We analyze the different cases of rotational surfaces depending on the value of $\lambda$.

\begin{enumerate}
\item  {\it Case $\lambda>1/4$.}

If $\lambda>1/4$, then $\theta'(0)=\lambda+1/2>0$ and this means that $\theta$ is increasing in a neighbourhood of $s=0$.  We know that  the point $Q_1=(\pi/2,1/(2\lambda))$ is a stable spiral point. Therefore, and by the phase portrait (Fig. \ref{fig41}, left), the angle function $\theta$ is increasing in a first moment, next $\theta$ crosses the value $\pi/2$, and next decreases crossing $\pi/2$ again. This behavior  is repeating as $\theta\rightarrow\pi/2$: see Fig. \ref{fig41}, right. On the other hand, the function $x$ is bounded with $x<1/\lambda$ and oscillating around the value  $1/(2\lambda)$ as $s\rightarrow \infty$, being this value its limit. Since $\theta\rightarrow\pi/2$, then $z'(s)=\sin\theta\rightarrow 1$ and the function $z$ increasing towards $\infty$.

 We have proved that the profile curve is an embedded curve converging to the vertical line $x=1/(2\lambda)$ and crossing this line infinitely times.

 \begin{figure}[hbtp]
\begin{center}
\includegraphics[width=.3\textwidth]{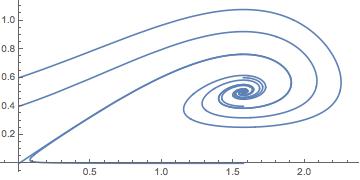}\qquad \includegraphics[width=.08\textwidth]{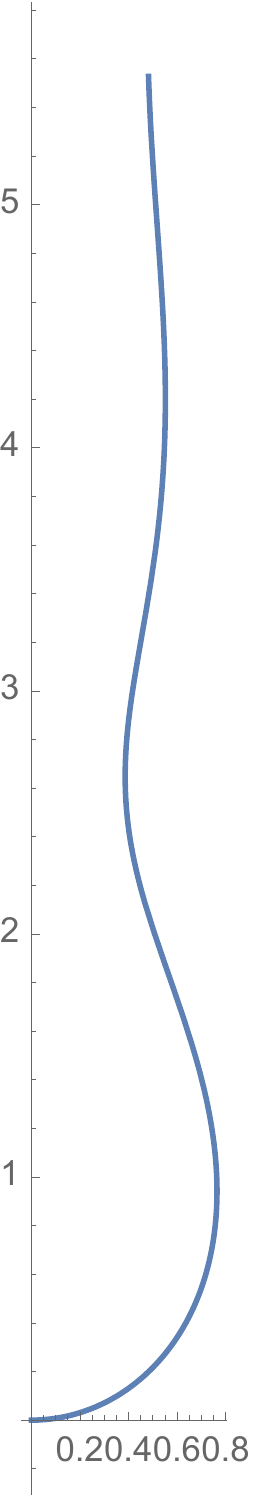}\end{center}
\caption{Rotational  surfaces for $\lambda>1/4$. Here $\lambda=1$. Left: the phase portrait around $Q_1$. Right: the profile curve intersecting    the rotational axis}\label{fig41}
\end{figure}

\item {\it Case $ \lambda=1/4$.}

  The singularity $Q_1$ is a stable improper node. As in the above case, $\theta$ is increasing for $s>0$  and  we have again  $\theta\rightarrow\pi/2$, $x(s)\rightarrow 1/(2\lambda)$ and $z(s)\rightarrow\infty$. The profile curve is embedded converging to the vertical line of equation $x=1/(2\lambda)$.

\item {\it Case $0<\lambda<1/4$.}

The function   $\theta$ is increasing again  in a neighbourhood of $s=0$.    Since $Q_1$ is a stable node,   the function $\theta$ does not attain the value $\pi/2$  (Fig. \ref{fig42}, left).  Then $\theta$ is increasing in its domain: on the contrary, at the first point $s_1$ where $\theta $ decreases, we have $\theta'(s_1)=0$ and $\theta''(s_1)\leq 0$, but 
\begin{equation}\label{segunda}
\theta''(s_1)=-\theta'(s_1)\left(\sin\theta(s_1)+\frac{\cos\theta(s_1)}{x(s_1)}\right)+\frac{\sin\theta(s_1)\cos\theta(s_1)}{x(s_1)^2}>0.
\end{equation}
 Thus $\theta$ is increasing with $\theta\rightarrow\pi/2$.   This proves that $\alpha$ is a graph on the $x$-line with $x(s)$ an increasing function and   $x(s)\rightarrow 1/(2\lambda)$. Since $\theta'(s)>0$, the graph of $\alpha$ is convex on the $x$-interval $[0,1/(2\lambda))$: see Fig. \ref{fig42}, right.

 \begin{figure}[hbtp]
\begin{center}
\includegraphics[width=.2\textwidth]{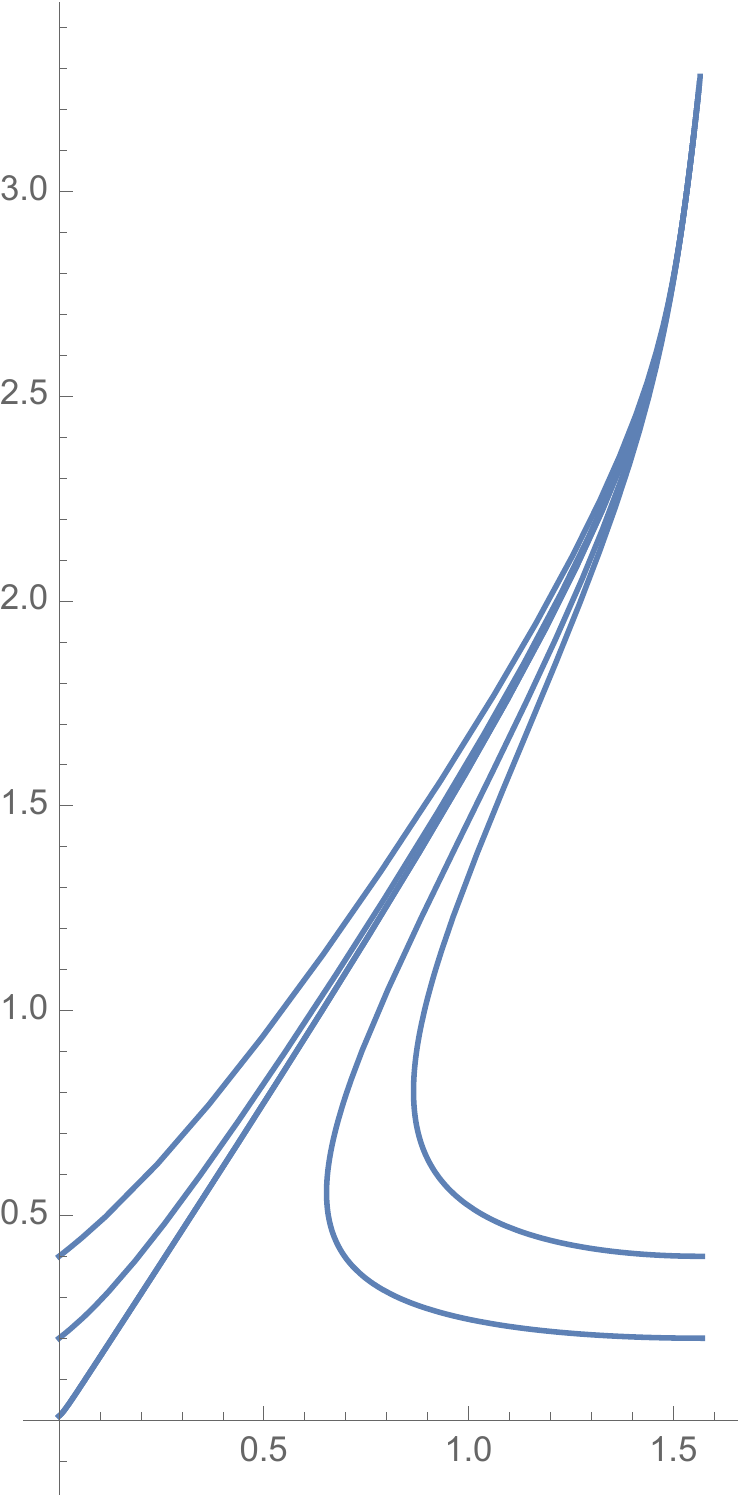}\qquad \includegraphics[width=.2\textwidth]{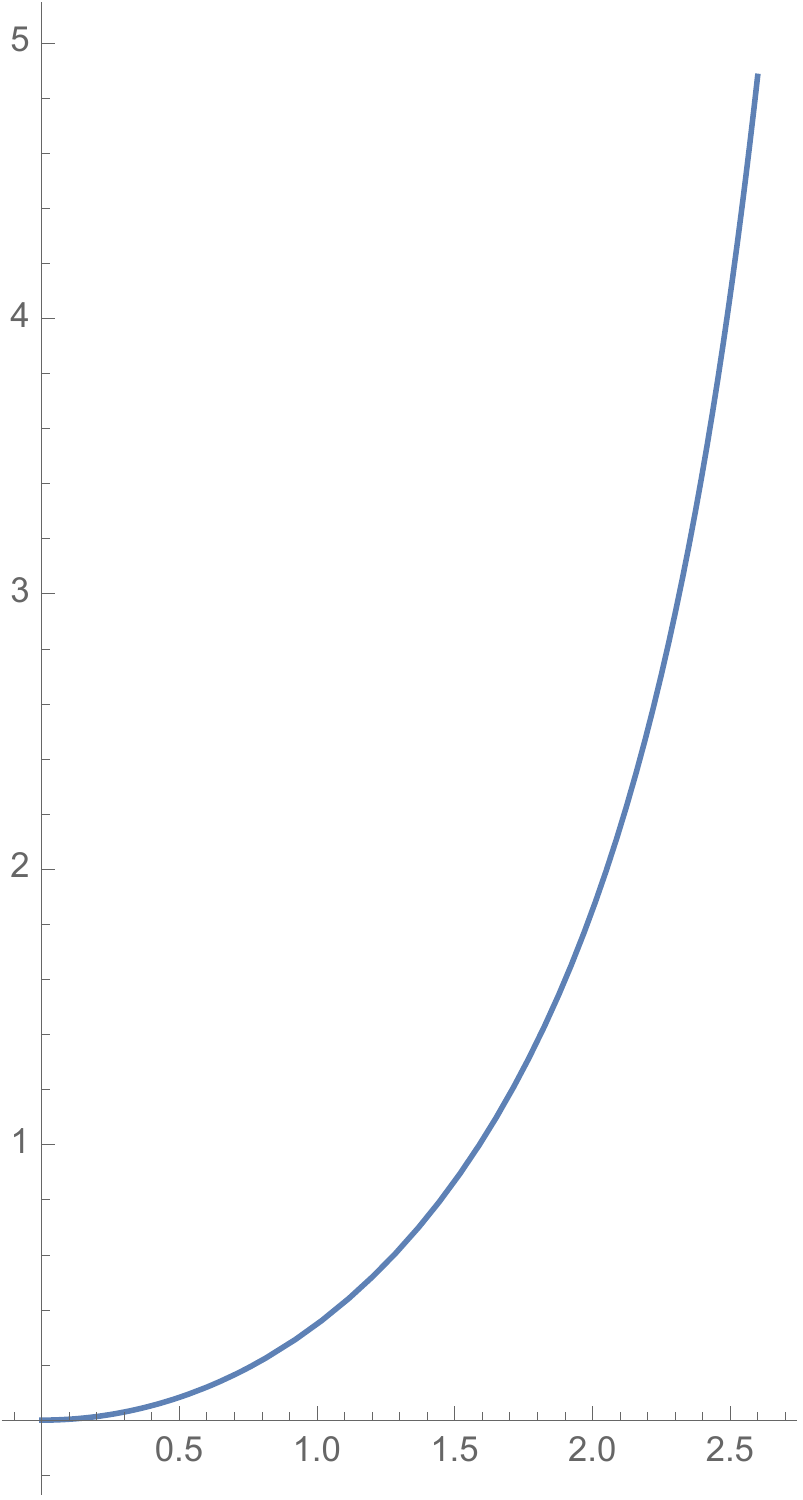}\end{center}
\caption{Rotational  surfaces for $0<\lambda<1/4$. Here $\lambda=0.15$. Left: the phase portrait. Right: the profile curve intersecting    the rotational axis}\label{fig42}
\end{figure}
\item {\it Case $-1/2<\lambda<0$. }

As $\theta'(0)=\lambda+1/2>0$, we know that  $\theta$ is increasing around $s=0$. We prove that $\alpha$ is a graph on the $x$-line. If there exists a first point $s_1>0$ such that $\alpha'(s_1)$ is vertical, then we have two possibilities. If $\theta(s_1)=\pi/2$, then $\theta'(s_1)\geq 0$, but  from (\ref{3}), we have $\theta'(s_1)=2\lambda-1/x(s_1)<0$. This implies that $\theta(s_1)=-\pi/2$ and $\theta'(s_1)\leq 0$, in particular, there exists $s_0\in (0,s_1)$ with $\theta(s_1)\in (0,\pi/2)$ such that
$\theta'(s_0)=0$ and $\theta''(s_0)\leq 0$. However, as in (\ref{segunda}), we have $\theta''(s_1)>0$. This contradiction proves that $\alpha$ is a graph and $\theta$ is as increasing function with $0<\theta(s)<\pi/2$ and since  $\theta'>0$, then  $\alpha$ is  convex. In particular, the singularity     $Q_2=(-\pi/2,-1/(2\lambda))$ is  not attained as  $s\rightarrow\infty$: see Fig. \ref{fig43}, left.

Finally, we show that function $x$ is not bounded (Fig. \ref{fig43}, right). As $x'(s)=\cos\theta(s)>0$, then $x=x(s)$ is increasing. If  $x$ is bounded, then $x'(s)\rightarrow 0$ and $x(s)\rightarrow \bar{x}$ for some positive number $\bar{x}$. Then  $\theta\rightarrow\pi/2$ as $s\rightarrow\infty$.   Letting  $s\rightarrow\infty$  in (\ref{3}),    we get $0=2\lambda-1/\bar{x}$, a contradiction because $\lambda<0$. Definitively, $\alpha$ is a convex graph on $(0,\infty)$.

 \begin{figure}[hbtp]
\begin{center}
\includegraphics[width=.4\textwidth]{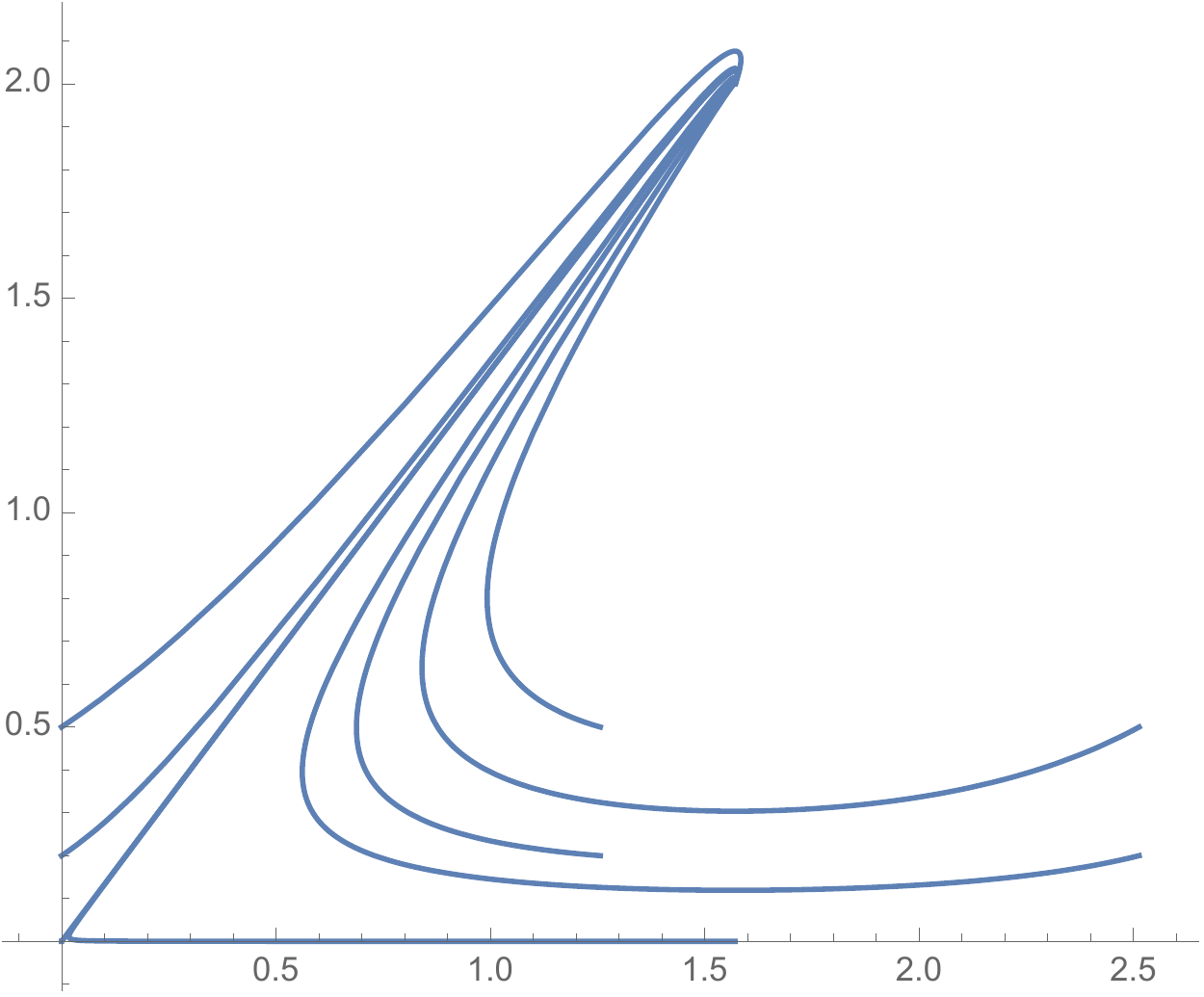}\qquad \includegraphics[width=.4\textwidth]{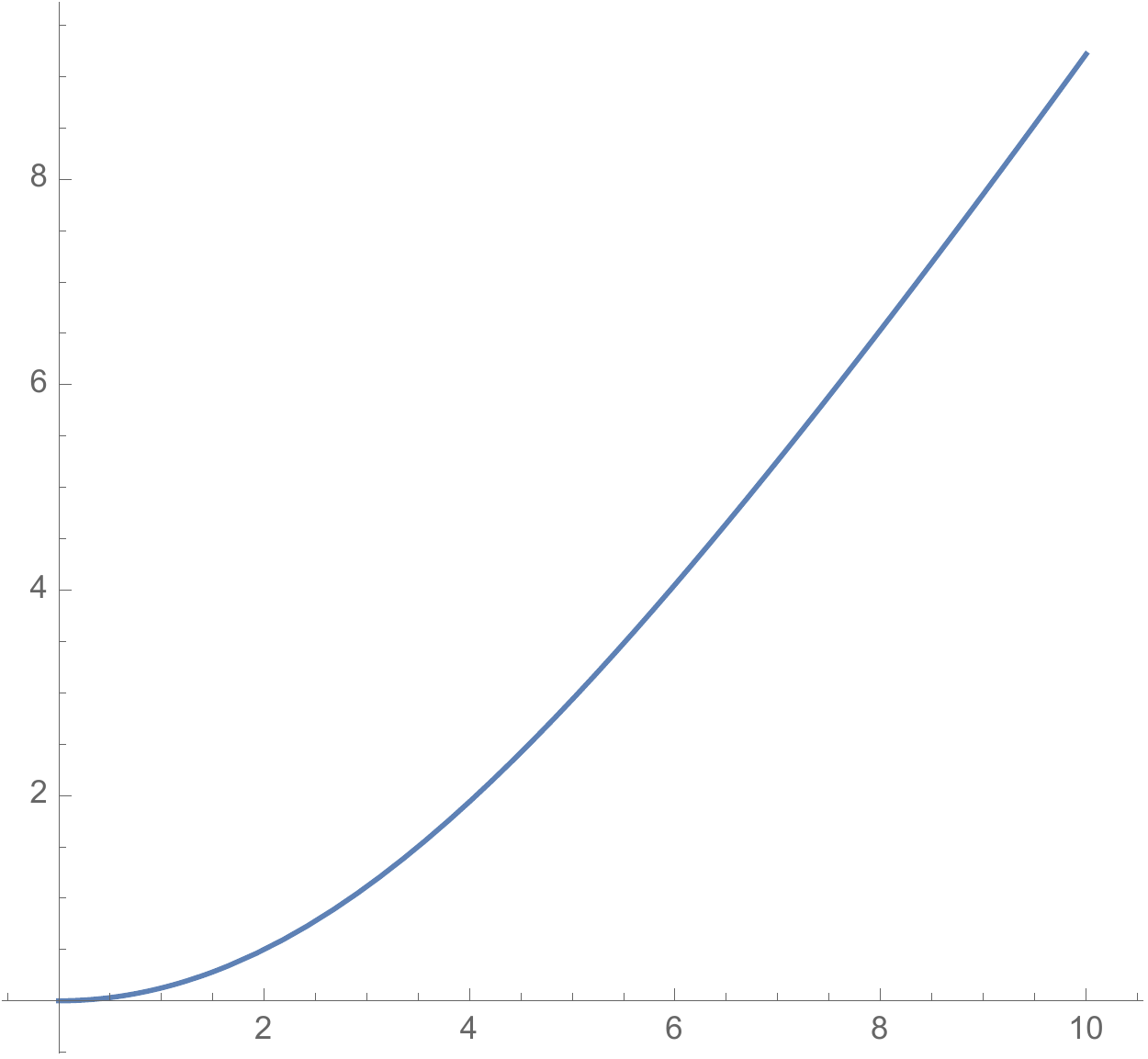}\end{center}
\caption{Rotational  surfaces for $-1/2<\lambda<0$. Here $\lambda=-0.25$. Left: the phase portrait around $Q_2$. Right: the profile curve intersecting    the rotational axis}\label{fig43}
\end{figure}

\item {\it Case $\lambda<-1/2$.}

As $\theta'(0)=2\lambda+1<0$, the function $\theta$ is initially decreasing, in particular, $x$ is increasing. Now the point  $Q_2=(-\pi/2,-1/(2\lambda))$ is an unstable spiral point with respect to it, the trajectories curl anti-clockwise towards infinity  (Fig. \ref{fig44}, left). In particular, we have $\theta(s)\rightarrow -\infty$ and $x(s)\rightarrow\infty$. This implies that $\alpha$ is a curve that self-intersects and curls clockwise to infinity (Fig. \ref{fig44}, right).

 \begin{figure}[hbtp]
\begin{center}
\includegraphics[width=.6\textwidth]{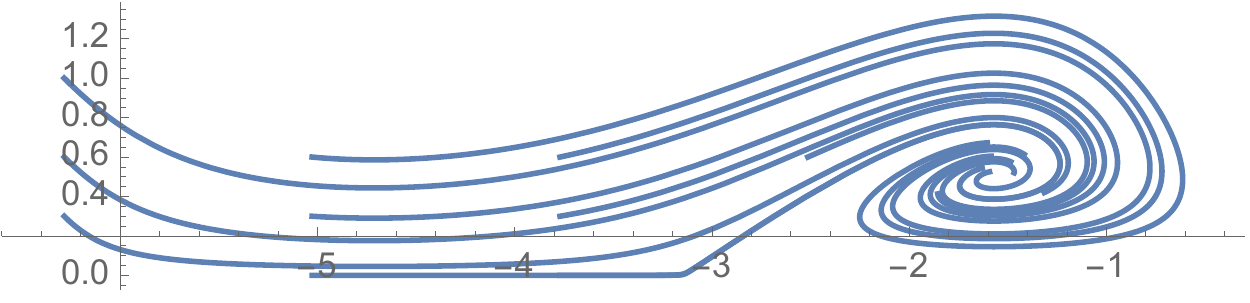}\quad \includegraphics[width=.3\textwidth]{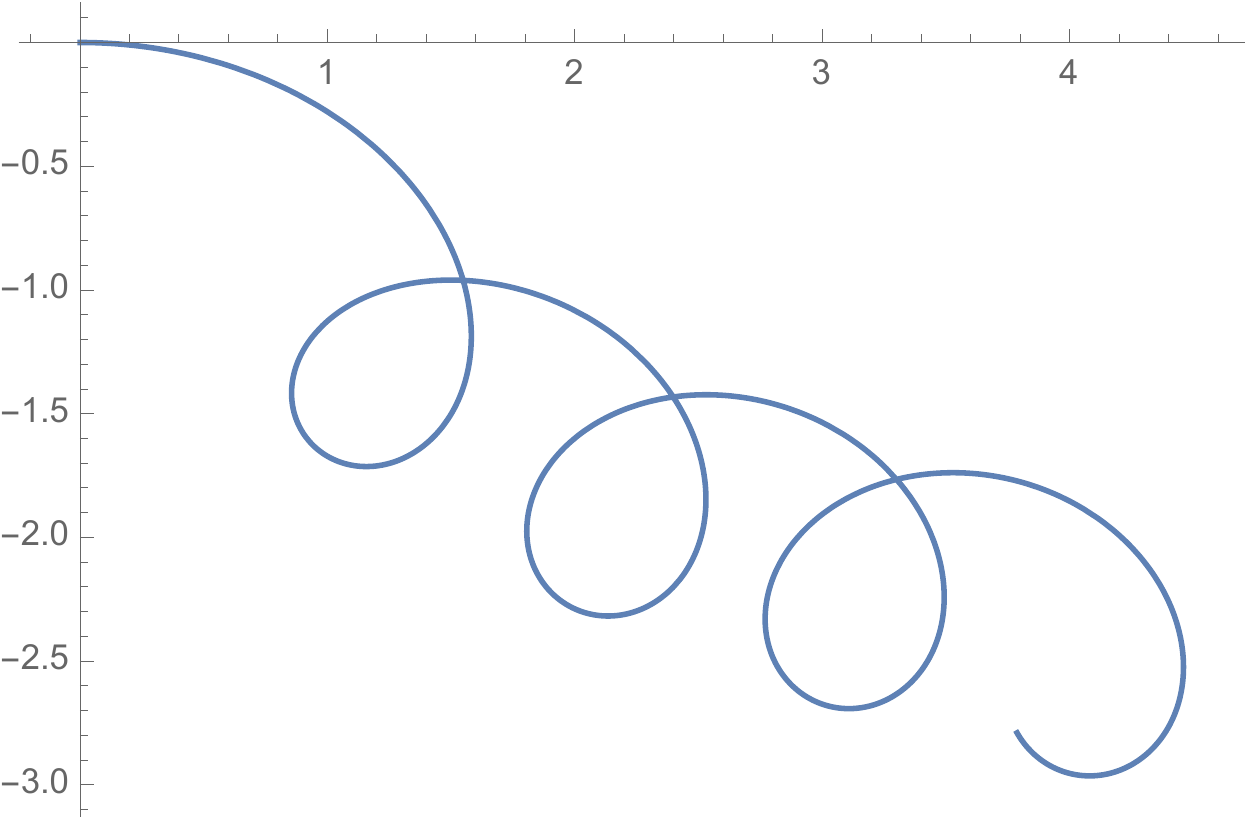}\end{center}
\caption{Rotational  surfaces for $ \lambda<-1/2$. Here $\lambda=-1$. Left: the phase portrait around $Q_2$. Right: the profile curve intersecting    the rotational axis}\label{fig44}
\end{figure}

\end{enumerate}

We summarize in the next result.

\begin{theorem}\label{trot} We have the next classification of rotational $\lambda$-translating solitons that intersect the rotational axis.   Without loss of generality, we suppose that the rotational axis is the $z$-axis and the density vector is $(0,0,1)$. Denote $C_\lambda$ the right vertical cylinder of radius $1/(2|\lambda|)$.

\begin{enumerate}
\item Case $\lambda\geq 1/4$. The surface is embedded and  asymptotic to $C_\lambda$.
\item Case $0<\lambda< 1/4$. The surface is a convex graph on a disc   of radius $1/(2\lambda)$ and asymptotic to $C_\lambda$.
\item Case $\lambda=0$. The surface is the bow soliton.
\item Case $-1/2<\lambda<0$. The surface is a convex entire graph on the $xy$-plane.
\item Case $\lambda=-1/2$. The surface is a horizontal plane.
\item Case $\lambda<-1/2$. The surface has infinity self-intersections.
\end{enumerate}

\end{theorem}

We end this section obtaining the classification of the  rotational $\lambda$-translating solitons that do not intersect the rotational axis. We consider the solution $(x,z,\theta)$ of (\ref{11}) with initial conditions $x(0)=x_0>0$ and $\theta(0)=\theta_0$. Since the solution is also defined for negative values of $s$, by Prop. \ref{pr26} the solution in the interval $(-\infty,0)$ coincides with the solution of (\ref{11}) for the value $-\lambda$ and initial conditions $x(0)=x_0$ and $\theta(0)=\theta_0+\pi$.  As $s\rightarrow\infty$, the behavior of the profile curve is as in Th. \ref{trot} according to the trajectories of  (\ref{pp}) . When $s\rightarrow-\infty$, the singularity $Q_2$ is attained when $\lambda<0$.   In the next result, we distinguish case-by-case depending on the value of $\lambda$, where the first statement refers when $s\rightarrow\infty$ in the profile curve (positive branch)   and the second statement   when $s\rightarrow-\infty$ (negative branch). Thus we deduce from Th. \ref{trot}:

\begin{theorem}\label{trot2}
We have the next classification of rotational $\lambda$-translating solitons that do not intersect the rotational axis.   Without loss of generality, we suppose that the rotational axis is the $z$-axis and the density vector is $(0,0,1)$. Denote $C_\lambda$ the right vertical cylinder of radius $1/(2|\lambda|)$.
\begin{enumerate}
\item For each $\lambda$, the cylinder $C_\lambda$.
\item Case $\lambda> 1/4$. The surface has two ends, one end is embedded and  asymptotic to $C_\lambda$ and the other end has infinity self-intersections.
\item Case $0<\lambda\leq1/4$. The surface has two embedded  ends,  one end is asymptotic to $C_\lambda$ and the other one is a convex   graph on the compliment of a round disc of the $xy$-plane.
\item Case $\lambda=0$. The surface is the winglike-surface translating soliton (\cite{css}).
\item Case $-1/2\leq\lambda<0$. The surface has two embedded ends, one end  is a convex   graph on the compliment of a round disc of the $xy$-plane and the other end is asymptotic to $C_\lambda$.
\item Case $\lambda<-1/2$. The surface has one end with infinity self-intersections and the other end   is  embedded and   asymptotic to $C_\lambda$.
\end{enumerate}

\end{theorem}
Some examples of rotational $\lambda$-translating solitons that do not intersect the rotational axis appear in Fig. \ref{fig45}
\begin{figure}[hbtp]
\begin{center}
\includegraphics[width=.2\textwidth]{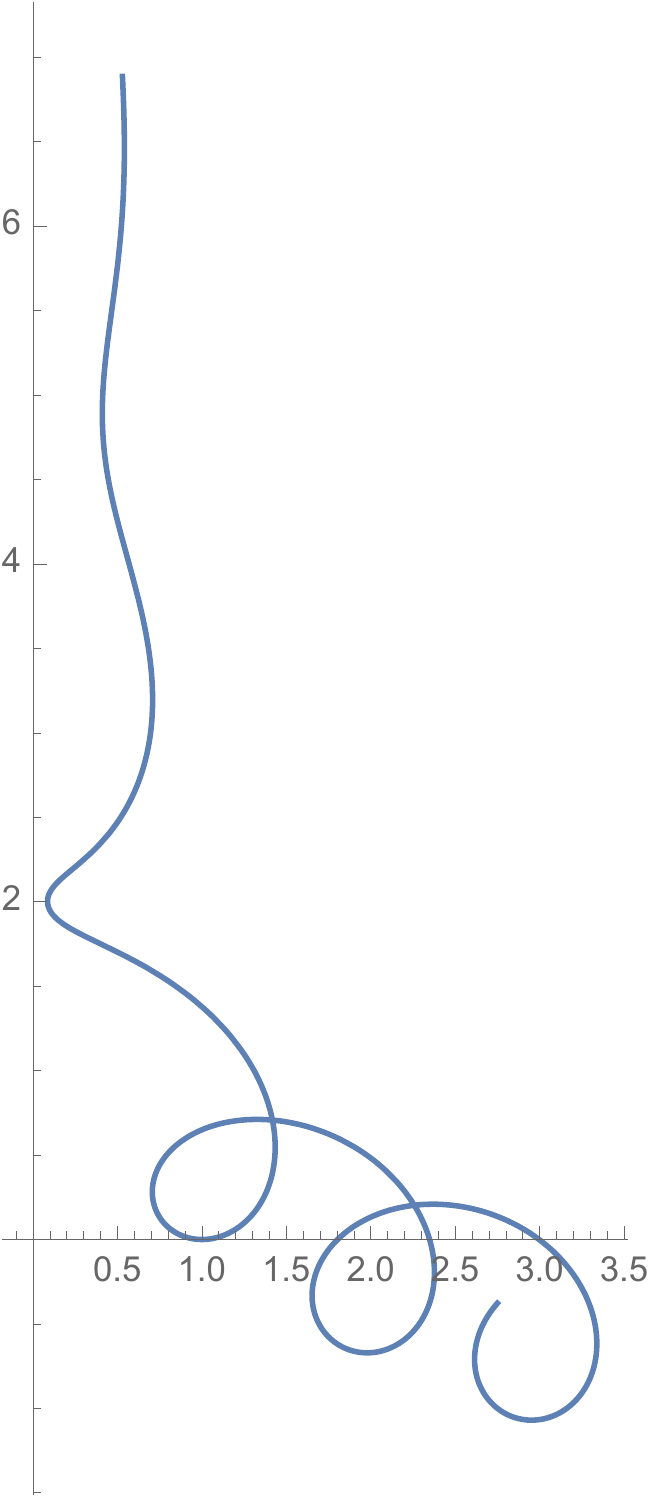}\qquad \includegraphics[width=.3\textwidth]{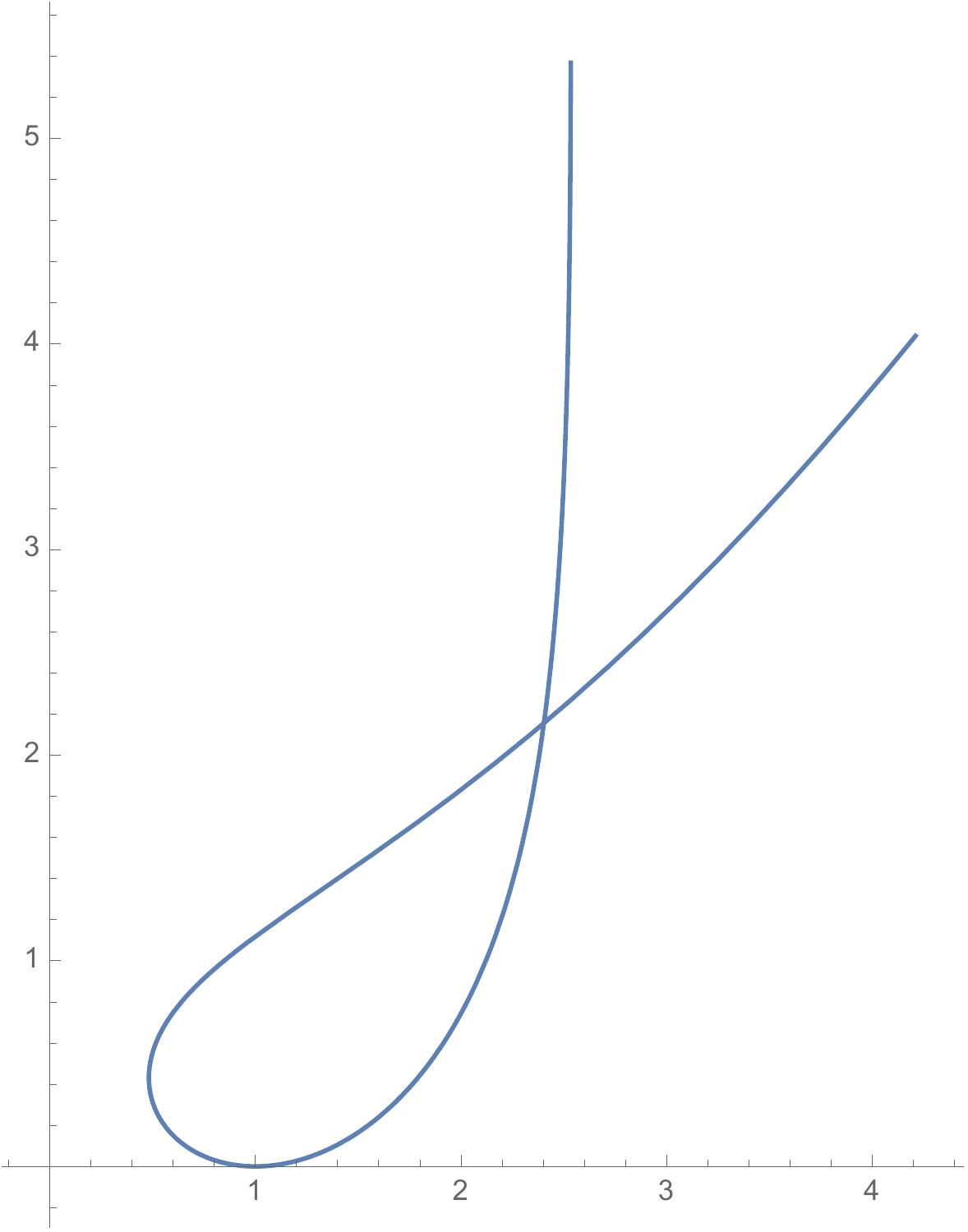}\quad \includegraphics[width=.34\textwidth]{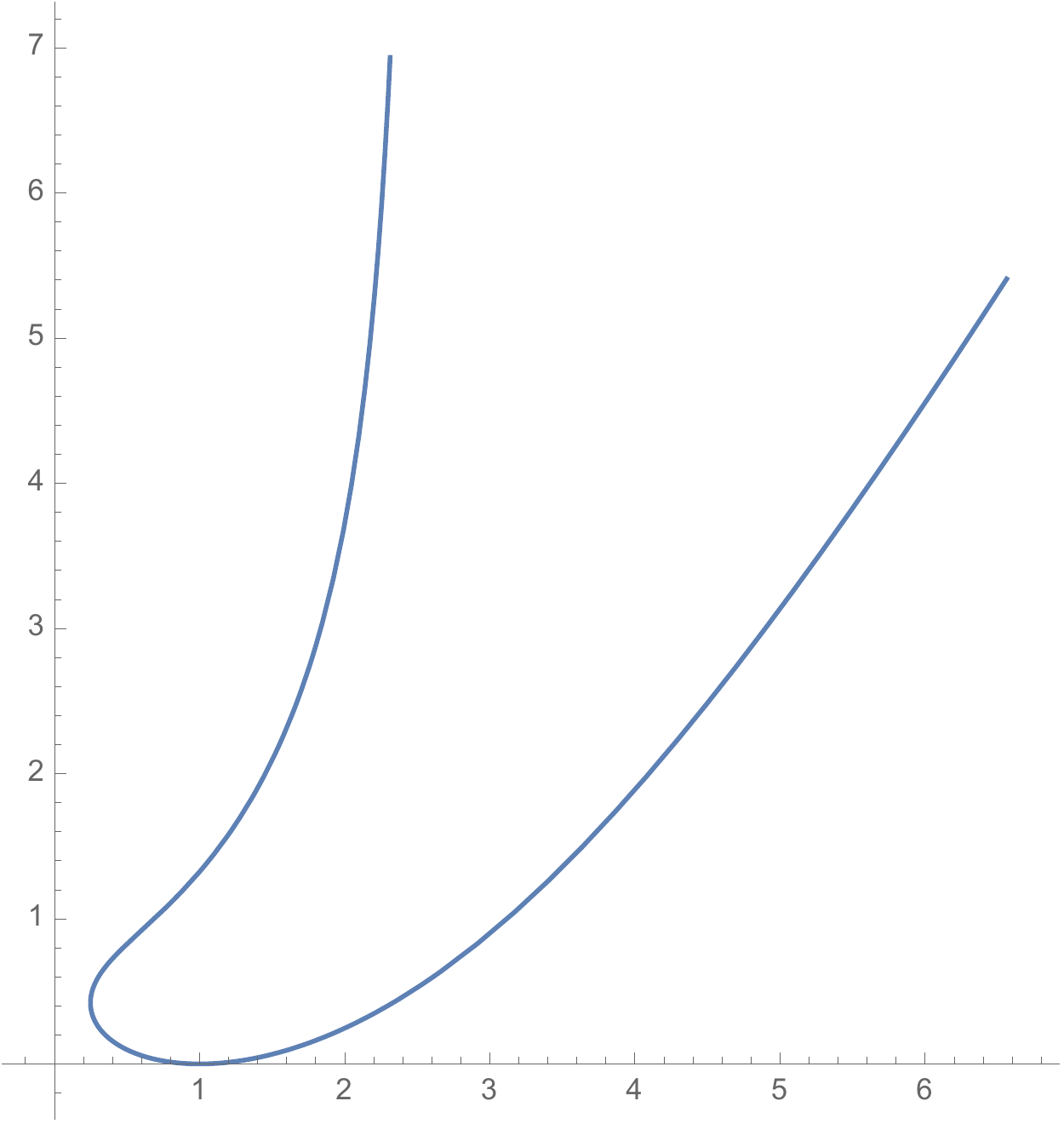}

\end{center}
\caption{Rotational  $\lambda$-translating solitons that do no intersect the rotational axis. Left: $\lambda=1$. Middle: $\lambda=0.2$. Right: An embedded rotational  $1/4$-translating soliton that does not intersect the rotational axis and   the initial condition is $\theta(0)=\pi$}\label{fig45}
\end{figure}

From the above classification we ask   if among these surfaces, and besides the cylinders $C_\lambda$,  there are embedded examples.  We have to consider  $\theta(0)=\theta_0$ inall its generality by taking any value  in an interval of length $2\pi$. If   $\lambda>1/4$ or $\lambda<-1/2$, then it is clear that  the angle function $\theta$ contains the interval $[-\pi,\pi]$, and these examples are not embedded.

If $0<\lambda\leq 1/4$, the negative branch says that $x(s)\rightarrow\infty$. Thus, by letting $s\rightarrow\infty$  in (\ref{3}), we have $0=2\lambda+\cos\theta_1$ with $\theta_1\in (-\pi,-\pi/2)$. When $\theta(0)=0$,  the range of $\theta$ is included in the interval $(\theta_1,\pi)$. If we take $\theta_0=\pi$, then we cover the rest of cases. In the first case, indeed, when $\theta(0)=0$, the positive branch meets the negative one by the range of $\theta$. In contrast,  if $\theta(0)=\pi$, the positive branch converging to the vertical line $x=1/(2\lambda)$ does not intersect the negative branch and we conclude that the profile curve is embedded: see Fig. \ref{fig45}, right.

A similar situation occurs when $-1/2<\lambda<0$ but now the embedded profile curve appears when $\theta(0)=0$. When $\lambda=-1/2$, the initial value $\theta(0)=0$ gives a horizontal plane, which meets the rotational axis, and if $\theta(0)=\pi$, then the profile curve is not embedded.

As a consequence, we have:
\begin{corollary} The only complete embedded rotational $\lambda$-translating solitons that are embedded are the cylinders $C_\lambda$ and some cases when  $-1/2\leq\lambda\leq 1/4$.
\end{corollary}

\section{A further result on rotational surfaces}\label{sriemann}

A generalization of  rotational surfaces are the  surfaces foliated by a one-parameter family of circles contained parallel planes: in case that the curve of the centers of the circles is a straight line orthogonal to each plane, then the surface is rotational. Our motivation in this section comes from the theory of minimal surfaces in Euclidean space where it is known that besides the rotational surfaces (plane and catenoid),  there exists a family of non-rotational minimal surfaces    foliated by circles in parallel planes, called in the literature the Riemann minimal examples (\cite{ri}). In this section we investigate the corresponding  problem for  $\lambda$-translating solitons. In view of Prop.\ref{pr1} we assume that the foliation planes are orthogonal to the density vector.

\begin{theorem} \label{t51}
If $\Sigma$ is a $\lambda$-translating soliton parametrized by a one-parameter family of circles contained in planes orthogonal to the density vector, then $\Sigma$ is a surface of revolution. \end{theorem}

\begin{proof} 	
	After a change of coordinates, we suppose that the density vector is $e_3=(0,0,1)$. Without loss of generality, we suppose that the curve of centers of circles is a graph on the $z$-axis, namely,  $s\rightarrow (a(s),b(s),s)$, $a,b\in C^\infty(\r)$, $I\subset\r$. Then a parametrization of $\Sigma$ is
$$X(s,t)=(a(s),b(s),s)+r(s)(\cos(t),\sin(t),0), \quad s\in I, t\in\r,$$
where $r(s)>0$. We observe that if the functions $a$ and $b$ are both constant, then $\Sigma$ is a surface of revolution about an axis parallel to $e_3$. We compute (\ref{eq1}) with the above parametrization $X(s,t)$. Let $\{E,F,G\}$ denote the first fundamental form of $\Sigma$ in coordinates with respect to $X$. If $(u,v,w)$ stands for the determinant of three vectors $u,v,w\in\r^3$, then Eq.  (\ref{eq1}) is
\begin{equation}\label{zz}
\frac{Z}{W^{3/2}}=2\lambda+\frac{(X_s,X_t,e_3)}{W^{1/2}},
\end{equation}
where $W=EG-F^2$ and
$$Z=E(X_s,X_t,X_{tt})-2F(X_s,X_t,X_{st})+G(X_s,X_t,X_{ss}).$$
We write (\ref{zz}) as
$$P=(Z-(X_s,X_t,e_3)W)^2-4\lambda^2 W^3=0,$$
and  $P$ is an expression of type
$$P(s,t)=\sum_{n=0}^6 A_n(s)\cos(nt)+B_n(s)\sin(nt),$$
for certain   functions $A_n, B_n\in C^\infty (I)$. Since the functions $\{\cos(nt),\sin(nt);0\leq n\leq 6\}$ are linearly independent, we conclude that $A_n=B_n=0$ for all $0\leq n\leq 6$.

The proof of Th. \ref{t51} is by contradiction. Suppose then that $\Sigma$ is not a surface of revolution, which means that the curve of centers is not a straight line. Thus there exists a subinterval of $I$,   which we rename by $I$ again, where $a'\not=0$ or $b'\not=0$. Without loss of generality, we suppose $b'\not=0$. The computation of $A_6$ and $B_6$ gives
$$A_6(s)=-\frac{1}{32} \left(4 \lambda ^2-1\right) r^6 \left(a'^6-b'^6-15 a'^4 b'^2+15 a'^2 b'^4\right),$$
$$B_6(s)=-\frac{1}{16} \left(4 \lambda ^2-1\right) r^6 a' b' \left(3 a'^4+3 b'^4-10 a'^2 b'^2\right).$$
We distinguish two cases:

\begin{enumerate}
\item Case $\lambda^2\neq 1/4$. Then $A_6=0$ is the equation $a'^6-b'^6-15 a'^4 b'^2+15 a'^2 b'^4=0$. By solving  for $a'^2$, we have $a'^2=b'^2$ or $a'^2=(7\pm 4\sqrt{3})b'^2$. Thus $a'=\pm mb'$ with $m^2=(7\pm 4\sqrt{3})$. By substituting into $B_6=0$, this equation reduces into $(3m^4-10m^2+4)b'^4=0$, a contradiction because $b'\not=0$ and the value of $m^2$.

\item Case $\lambda^2=1/4$. The first non trivial coefficients of $P$ are $A_4$ and $B_4$.
\begin{enumerate}
\item Sub-case that $r$ is a constant function in $I$, that is, $r'=0$. Then the linear combination of $A_4$ and $B_4$ given by $(a'^3-3a'b'^2)A_4-(b'^3-3a'^2b')B_4=0$ simplifies into $(a'^2+b'^2)^3(a'-2a'')=0$. Then
$a''=a'/2$, obtaining $a(s)=m_1e^{s/2}+n_1$, $m_1,n_1\in\r$. Then $A_4=0$ and $B_4=0$ are now
$$(b'-2b'')(3m_1^2e^s-4b'^2)=0,\quad (b'-2b'')(m_1^2 e^s-12b'^2)=0.$$
Since $b'\not=0$,  we deduce $b'-2b''=0$, hence $b(s)=m_2 e^{s/2}+n_2$, $m_2\not=0$. Then $A_3=0$ and $B_3=0$ write as
$$m_1(m_1^2-3m_2^2)(4+e^s(m_1^2+m_2^2)=0,$$
$$m_2(-3m_1^2+m_2^2)(4+e^s(m_1^2+m_2^2)=0.$$
We conclude $m_1=m_2=0$, that is, $b'=0$,  a contradiction.
\item Sub-case $r'\not=0$. After a computation, the non trivial linear combination 
$$  4   a' \left(b'^3-a'^2 b'\right)A_4+ (-6 a'^2 b'^2+a'^4+b'^4)B_4=0$$
 simplifies into
 $$r^6(a'^2+b'^2)^3(a'b''-a''b')=0.$$
 If $a'\not=0$ at some point $s_0\in I$, then $a'b''-a''b'=0$ and thus $a'(s)=mb'(s)$ for some constant $m>0$. Using $b'\not=0$, the equations $A_4=0$ and $B_4=0$ now yield
 $$(m^4-6m^2+1)(2rb''-b'(r+4r'))=0,\quad (m^2-1)(2rb''-b'(r+4r'))=0.$$
 Hence $2rb''-b'(r+4r')=0$. Then $b''=b'(r+4r')/(2r)$ and putting into $A_3=0$ and $A_2=0$, we have
 $$q_1:=2 \left(m^2+1\right) b'^2+2 r'^2+r \left(r'-2 r''\right)+2=0,$$
 $$q_2:=8 r' \left(\left(m^2+1\right) b'^2+r'^2+1\right)+r \left(4 r' \left(r'-2 r''\right)+1\right)=0.$$
 Finally, the linear combination $4r'q_1-q_2=0$ writes simply as  $r=0$, a contradiction. This contradiction proves $a'=0$ in $I$. In such a case, $A_4=r^5b'^3(2rb''-b'(r+4r'))/8$ and thus $2rb''=b'(r+4r')$. A similar argument as above concludes that $4r'q_1-q_2=r=0$, a contradiction

\end{enumerate}
\end{enumerate}

\end{proof}


\begin{thebibliography}{11}


\bibitem{aw} S. J. Altschuler, L. F. Wu, Translating surfaces of the non-parametric mean curvature flow with prescribed
contact angle. Calc. Var.  2 (1994),  101--111.




\bibitem{css} J. Clutterbuck, O. Schn\"{u}rer,  F. Schulze, Stability of translating solutions to mean curvature flow. Calc. Var.  29 (2007), 281-- 293.

\bibitem{go1} J. M. Gomes, Sobre hypersuperficies de curvatura media constante no espaco hiperb\'olico. Tese de Doutorado, IMPA, 1984.

\bibitem{go2} J. M. Gomes, Spherical surfaces with constant mean curvature in hyperbolic space, Bol. Soc. Bras. Math. 18 (1987), 49--73.


 \bibitem{gr} M. Gromov, Isoperimetry of waists and concentration of maps. Geom. Func. Anal 13 (2003), 178--215.

\bibitem{ha} H. P. Halldorsson, Helicoidal surfaces rotating/translating under the mean
curvature flow. Geom. Dedicata 162 (2013), 45--65.

\bibitem{hs} G. Huisken, C. Sinestrari,  Mean curvature flow singularities for mean convex surfaces. Calc. Var.  8 (1999), 1--14.

\bibitem{il} T. Ilmanen, Elliptic regularization and partial regularity for motion by mean curvature. Mem. Amer. Math. Soc. 108 (1994).

\bibitem{lo2} R. L\'opez, Minimal surfaces in Euclidean space with a log-linear density.	arXiv:1410.2517 [math.DG]

\bibitem{mar} F. Mart\'{\i}n, A. Savas-Halilaj,  K. Smoczyk, On the topology of translating solitons of the mean curvature flow. Calc.  Var. 54  (2015), 2853--2882.

\bibitem{mh} N. Minh, D. T. Hieu, Ruled minimal surfaces in $R^3$ with density $e^z$. Pacific J. Math.  243  (2009), 277--285.


\bibitem{mo} F.  Morgan, Manifolds with density. Notices Amer. Math. Soc. 52 (2005), 853--858.

\bibitem{ri}  B. Riemann, \"{U}ber die Fl\"{a}che vom kleinsten Inhalt bei gegebener Begrenzung. Abh.  K\"{o}nigl, d. Wiss. G\"{o}ttingen, Mathem. Cl., 13: 3--52, 1867. K. Hattendorf, editor. JFM 01.0218.01.

\bibitem{se} J. B. Serrin, The problem of Dirichlet for quasilinear elliptic differential equations with many independent variables. Phil. Trans. R. Soc. Lond. 264 (1969), 413--496.

\bibitem{sh} L. Shahriyari, Translating graphs by mean curvature flow. Geom. Dedicata, 175 (2015),  57--64.

\bibitem{sm} G. Smith, On complete embedded translating solitons of the mean curvature flow that area of finite genus,  arXiv:1501.04149 [math.DG], 2015.

\bibitem{wa} X-J. Wa, Convex solutions to the mean curvature flow. Ann. Math. 173 (2011), 1185--1239.


\end{thebibliography}
\end{document}